\documentclass[a4paper,10pt]{article}
\usepackage[utf8x]{inputenc}
\usepackage{amsmath}
\usepackage{amssymb}
\usepackage{hyperref}
\usepackage{enumerate}

\usepackage{amsthm}
\usepackage[draft]{todonotes}

\usepackage{epstopdf}
\usepackage{pdfpages}

\usepackage{framed}

\usepackage[margin=1.2in]{geometry}
\usepackage{cancel,amsthm,amssymb,amsmath,mathrsfs}
\usepackage{color,tikz}
\usetikzlibrary{patterns, decorations.markings,arrows, calc}

\renewcommand{\(}{\left(}
\renewcommand{\)}{\right)}
\newcommand{\ol}{\overline}

\newcommand{\e}{\mathrm{e}}
\renewcommand{\i}{\mathrm{i}}
\renewcommand{\i}{{i}}

\renewcommand{\d}{\mathrm{d}}

\newtheorem{RHP}{Riemann-Hilbert problem}

\newtheorem{rem}{Remark}
\newtheorem{Crem}[rem]{Conjectural and non rigorous Remark}
\newtheorem{thrm}{Theorem}
\newtheorem{lemma}[thrm]{Lemma}
\newtheorem{cor}[thrm]{Corollary}

\usepackage{color}

\renewcommand{\P}{\mathbf{P}}
\newcommand{\M}{\mathbf{M}}
\newcommand{\Q}{\mathbf{Q}}
\newcommand{\N}{\mathbf{N}}
\newcommand{\J}{\mathbf{J}}

\newcommand{\1}{\mathbf{1}}
\newcommand{\s}{\pmb{\sigma}}
\newcommand{\Z}{\mathbf{Z}}

\title{On the solution of the Zakharov-Shabat system, which arises in the analysis of the largest real eigenvalue in the real Ginibre ensemble}
\author{A.Minakov,\\ UCLouvain, IRMP, Belgium}
\date{}

\begin{document}

\maketitle

\begin{abstract}
Let $\lambda_{max}$ be a shifted maximal real eigenvalue of a random $N\times N$ matrix with independent $N(0,1)$ entries (the `real Ginibre matrix') in the $N\to\infty$ limit.

It was shown by Poplavskyi, Tribe, Zaboronski  \cite{PZT} that the limiting distribution of the maximal real eigenvalue has $s\to-\infty$ asymptotics
$$\mathbb{P} [ \lambda_{max} < s ] = \e^{\frac{1}{2\sqrt{2\pi}}\zeta(\frac32)s+\mathcal{O}(1)},$$
where $\zeta$ is the Riemann zeta-function.

This limiting distribution was expressed by Baik, Bothner \cite{BB18} in terms of the solution $q(x)$ of a certain Zakharov-Shabat inverse scattering problem, and the asymptotics was extended 
to the form
$$\mathbb{P} [ \lambda_{max} < s ] = \e^{\frac{1}{2\sqrt{2\pi}}\zeta(\frac32)t}c(1+_\mathcal{O}(1)),\ s\to-\infty.$$

We show that $q(x)$ is a smooth function, which behaves as $\frac{1}{x}$ as $x\to-\infty.$
Second, we show that the error term in the asymptotics is subexponential, i.e. smaller that $\e^{-C|s|}$ for any $C.$
Third, we identify the constant $c$ as a conserved quantity of a certain fast decaying solution $u(x,t)$ of the Korteweg-de Vries equation. This, in principle, gives a way to determine $c$ via the known long-time $t\to+\infty$ asymptotics of $u(x,t).$ We also conjecture a representation for the $c$ in terms of an integral of the Hastings-MacLeod solution of Painlev\'e II equation.
\end{abstract}

\tableofcontents

\section{Introduction} 
For $\gamma\in[0,1],$ define\footnote{The $r$ from \cite{BB18} equals $\i R(k;\gamma).$}
 \begin{equation}\label{R}R(k;\gamma) =  -\sqrt{\gamma} \e^{-k^2/4},\end{equation}
and consider the following Riemann-Hilbert problem (RHP):
\footnote{The $\mathbf{X}$ from \cite{BB18} equals $\e^{\pi\i\s_3/4}\M\e^{-\pi\i\s_3/4},$ where $\s_3=diag[1,-1]$.}

\begin{RHP}\label{RightRHP}
To find a $2\times2$ matrix-valued function $\M = \M(x,t;k;\gamma)$ that satisfies the following properties:
\begin{itemize}
\item analyticity: $\M(x,t;k;\gamma)$ is analytic in $k\in \mathbb{C}\setminus{\mathbb{R}},$
and continuous up to the boundary $k\in\mathbb{R};$
\item jumps:
 $\M_-=\M_+ \J_{\M},$ where
\[\J_\M=\begin{pmatrix}1 & \ol{R(k;\gamma)}\cdot\e^{-2\i \widehat\theta(x,t;k)} \\ -R(k;\gamma)\e^{2\i \widehat\theta(x,t;k)} & 1 -|R(k;\gamma)|^2\end{pmatrix},
 \ k\in\mathbb{R},\]
 where $\widehat\theta(x,t;k) = k x +4 k^3 t;$
\item asymptotics at the infinity:
\[\M(k)\to \mathbf{1}\equiv\begin{pmatrix}1&0\\0&1\end{pmatrix}\quad \mbox{ as } \quad k\to\infty.\]
\end{itemize}
\end{RHP}
\noindent Define the functions $q(x,t;\gamma), u(x,t;\gamma)$ by the formulas
\footnote{The $y(x;\gamma)$ from \cite{BB18} equals $y = \i q|_{t=0}.$}
\begin{equation}\label{qu}\begin{split}&
q(x,t;\gamma) = -2\i\lim\limits_{k\to\infty} k(\M(x,t;k;\gamma)-\mathbf{1})_{12}
=
2\i\lim\limits_{k\to\infty}  k(\M(x,t;k;\gamma)-\1)_{21}
\in\mathbb{R},
\\
&
u(x,t;\gamma) = q^2(x,t;\gamma)-q_x(x,t;\gamma),\quad 
\int_{x}^{+\infty}u(x,t;\gamma) = q(x,t;\gamma)+\int_x^{+\infty}q^2(z,t;\gamma)\d z,
\end{split}
\end{equation}
where the subscript $_{12}$ means the element situated on the intersection of the first row and the second column in the matrix. For $t=0,$ we denote $q(x;\gamma):=q(x,0;\gamma),$ $u(x;\gamma):=u(x,0;\gamma).$
The $q(x,t;\gamma)$ satisfies the (defocusing) modified Korteweg-de Vries equation (MKdV) and $u(x,t;\gamma)$ satisfies the Korteweg-de Vries equation,
\begin{subequations}\label{MKdVKdV}
\begin{align}
&q_t-6q^2q_x+q_{xxx}=0,
\\&
u_t-6uu_x+u_{xxx}=0,
\end{align}\end{subequations}
and $R(k;\gamma)$ is the reflection coefficient, associated with $q$ via MKdV scattering problem, and is the reflection coefficient, associated with $u$ via KdV scattering problem (see sections \ref{sectKdV}, \ref{sectMKdV} for a short explanation what does it mean).
Moreover, $q(x,t;\gamma)$ for $\gamma\in[0,1)$ is an example of a classical solution of MKdV, which is exponentially decaying as $x\to\pm\infty$ for all times $t.$ 
The $u(x,t;\gamma)$ is such an example for KdV, but already for all $\gamma\in[0,1],$ including the case $\gamma=1.$

\noindent Define the function
\[
F(2s;\gamma) = 
\exp\left[
-\frac12\int\limits_{s}^{+\infty}(z-s)q^2(z,0;\gamma)\d z\right]
\sqrt{\cosh(\sigma(s;\gamma))-\sqrt{\gamma} \sinh(\sigma(s;\gamma))},
\]
where
\footnote{The $\mu(.;\gamma)$ from \cite{BB18} equals $\mu(2s;\gamma) = \sigma(s;\gamma).$}
\[
\sigma(s;\gamma):=\int_s^{+\infty}q(x,0;\gamma)\d x.\]

\noindent 
For $\gamma = 1$ the above expression equals
\begin{equation}\label{F}\begin{split}
F(2s;1)
&=
\exp\left[-\frac12\int\limits_{s}^{+\infty}(z-s)q^2(z,0;1)\d z-\frac12\int\limits_{s}^{+\infty}q(z,0;1)\d z\right]
=
\\
&=
\exp\left[-\frac12\int\limits_{s}^{+\infty}
\(q(z;1)+\int\limits_{z}^{+\infty}q^2(x,0;1)\d x\)
\d z
\right]=
\\
&=\exp\left[-\frac12\int\limits_{s}^{+\infty}(z-s)u(z,0;1)\d z\right] = \exp\left[-\frac12\int\limits_{s}^{+\infty}\int\limits_{x}^{+\infty}u(z,0;1)\d z\d x\right].
\end{split}
\end{equation}

It was shown in \cite{BB18} that the function $F(s;1)$ with $\gamma = 1$ plays an important role in the analysis of real eigenvalues in the real Ginibre ensemble. Namely,

\begin{thrm} (Baik, Bothner, \cite{BB18})
Let $\left\{z_j(\mathbf{X})\right\}_{j=1}^n$ denote the eigenvalues of a $n\times n$ matrix with independent $N(0,1)$ entries (the `real Ginibre matrix').
Then 
\begin{equation}\label{PTZ}
\lim\limits_{n\to\infty}\mathbb{P}\(\max\limits_{j:z_j\in\mathbb{R}}z_j(\mathbf{X})\leq \sqrt{n}+s\) = F(s;1),\quad s\in\mathbb{R}.\end{equation}
\end{thrm}
The work  \cite{BB18} is based on a previous work of Rider, C. Sinclair \cite{RS}; Poplavsky, Tribe, Zaboronski \cite{PZT}, 
 where the left-had-side of \eqref{PTZ} is identified with a certain Fredholm determinant.

It was noticed in \cite{BB18} that for $\gamma\in[0,1)$ the function
$q(x;\gamma) := q(x,0;\gamma)$ belongs to the Schwartz class $\mathbb{S}(\mathbb{R})$, while for $\gamma = 1$ it does not.
Our first goal here is to answer the following question:
{\it to which class does $q(x;1)$ belong?}
We show that $q(x;1)$ is infinitely smooth in $x,$ decays exponentially as $x\to+\infty,$ and decays as $x^{-1}$ for $x\to-\infty$ (see formulas \eqref{q1}, \eqref{intq1} below).
In more details, we show 
\begin{thrm}\label{teor1}
\begin{enumerate}[(a)]
\item For any $\gamma\in[0,1],$ the function $q(x;\gamma)\in C^{\infty}(x\in\mathbb{R});$
\item For any $x\in\mathbb{R},$ the function $q(x;\gamma) \in C(\gamma\in[0,1]);$
\item for fixed $\gamma\in(0,1)$ and $x\to-\infty,$ for any $C>0,$
\begin{equation}\label{qgamma}\begin{split}&q(x;\gamma) = 
\dfrac{8 \kappa_{\gamma}^2\e^{2 x \kappa_{\gamma}}{L_{-1}(\gamma)}}{4\kappa_{\gamma}^2 - \e^{4x\kappa_{\gamma}}{L_{-1}(\gamma)}^2}
+\mathcal{O}(\e^{-C |x|}),
\\
&
\int\limits_{x}^{+\infty}q^2(\tilde x;\gamma)\d\tilde x = 2T_1(\gamma)- \dfrac{4\kappa_{\gamma}\e^{4 x \kappa_{\gamma}}{L_{-1}(\gamma)}^2}{4\kappa_{\gamma}^2 - \e^{4x\kappa_{\gamma}}{L_{-1}(\gamma)}^2}+\mathcal{O}(\e^{-C |x|}).
\end{split}
\end{equation}
Here $\kappa_{\gamma} = \sqrt{-2\ln\gamma}\geq0$ (so that $\kappa_1 = 0$ ), and 
\begin{equation}\label{T1}
\begin{split}&
T_1(\gamma) =  \frac{-1}{2\pi}\int\limits_{-\infty}^{+\infty}\ln(1-\gamma\e^{-\frac{s^2}{2}})\d s
=
\frac{1}{\sqrt{2\pi}}Li_{\frac32}(\gamma)
>0,
\\
&
L_{-1}(\gamma) = \frac{1}{\kappa_{\gamma}}
\exp\left[\dfrac{1}{\pi\i}\displaystyle\int\limits_{-\infty}^{+\infty}
\frac{\ln(1-\gamma\e^{-\frac{s^2}{2}})\ \d s}{s-\i\kappa_{\gamma}}\right]>0.
\end{split}
\end{equation}
\item 
for $\gamma=1,$ as $x\to-\infty,$ for any $C>0,$
\begin{equation}\label{q1}\begin{split}
&q(x;1)=\frac{2}{-2x+L_1(1)}+\mathcal{O}(\e^{-C|x|}),
\\
&
\int\limits_{x}^{+\infty}q^2(\tilde x;\gamma)\d\tilde x = 2T_1(1) + \frac{2}{2x-L_1(1)}+\mathcal{O}(\e^{-C|x|}),
\end{split}\end{equation}
where $T_1(1)$ is as in \eqref{T1}, 
\[T_1(1) =  \frac{-1}{2\pi}\int\limits_{-\infty}^{+\infty}\ln(1-\e^{-\frac{s^2}{2}})\d s
=
\frac{1}{\sqrt{2\pi}}Li_{\frac32}(1) = \frac{1}{\sqrt{2\pi}}\zeta\(\frac32\)\approx 1.042\,186\,978\,869 ,\]
and
\[L_1(1) = 2-\frac{1}{\pi}\int\limits_{\Sigma_{1/4}}{\ln\([1- \e^{-\frac{s^2}{2}}]\frac{s^2+1}{s^2}\)} \, \frac{\d s}{s^2} \approx  1.165\,194\,315\,878\,021\,340\,410\,354,\]
with the integral over the oriented contour $\Sigma_{1/4} = (-\infty,-\frac14)\cup(-\frac14,-\frac{\i}4)\cup(-\frac{\i}4,\frac14)\cup(\frac14,+\infty).$  

\item 
\begin{equation}\label{intq1}\int\limits_{x}^{+\infty}q(\tilde x;1)\d\tilde x = 
\ln\(-2x+L_1(1)\)+\frac12\ln2+\mathcal{O}(\e^{-C|x|}),\quad x\to-\infty.\end{equation}

\end{enumerate}
\end{thrm}

\noindent Denote, 
\begin{subequations}\label{cons}
\begin{align}
&\label{K}
\mbox{for $\gamma\in[0,1],\ $}\quad
H(\gamma) = 3\int\limits_{-\infty}^{+\infty}u^2(x,t;\gamma)\d x,
\quad
K(\gamma) = \int\limits_{-\infty}^{+\infty}xu(x,t;\gamma)\d x+H(\gamma)t;
\\
&
\mbox{for $\gamma\in[0,1],$ } \quad
N(\gamma)=3\int\limits_{-\infty}^{+\infty}\(q^4(x,t;\gamma)+q_x^2(x,t;\gamma)\)\d x>0;
\\
&
\mbox{for $\gamma\in(0,1),$ } \quad
M(\gamma) = \int\limits_{-\infty}^{+\infty}
xq^2(x,t;\gamma)\d x + N(\gamma)t - \frac12\ln|\ln\gamma|
-\ln2;
\\
&
\nonumber
\mbox{for $\gamma=1$ and $x<0,$ } \quad
\\
&\label{M1}
\qquad
M(1) = \int\limits_{-\infty}^{x}
\(zq^2(z,t;\gamma)-\frac1{z}\)\d z
+
\int\limits_{x}^{+\infty}
zq^2(z,t;\gamma)\d z
 + N(1)t+\ln|x|+\frac32\ln2-1.
\end{align}
\end{subequations}

All the functions $K(\gamma)$, $H(\gamma)$, $M(\gamma),$ $N(\gamma)$ are conserved quantities, i.e. they do not depend on time $t$. Quantity $M(1)$ do not depend on the choice of $x<0$ (Lemma \ref{lem_cons}).

\begin{cor}\label{cor1}
We have as $s\to-\infty,$ for any $C>0,$
\[
\begin{split}
&
F(s;1) = \e^{\frac12 T_1(1) s}\e^{-\frac12 K(1)}(1+\mathcal{O}(\e^{-C|s|})),
\end{split}
\]
and $K(1) = M(1).$
\end{cor}
\begin{proof}
Using asymptotics \eqref{q1}, \eqref{intq1} and conservation law \eqref{M1} at the time $t=0,$
we find 
\[\begin{split}&
-\frac12\int\limits_{s}^{+\infty}zq^2(z;1)\d z = \frac{s}{L_1-2s}+\frac12\ln(L_1-2s)-\frac12M(1)+\frac14\ln2+\mathcal{O}(\e^{-C|s|}),
\\
&
\frac{s}{2}\int\limits_{s}^{+\infty}q^2(z;1)\d z = s T_1(1)+\frac{s}{2s-L_1(1)}+\mathcal{O}(\e^{-C|s|}),
\\
&
-\frac12\int\limits_{s}^{+\infty}q(z;1)\d z
=
-\frac12\ln(L_1(1)-2x)-\frac14\ln2+\mathcal{O}(\e^{-C|x|}).
\end{split}\]
Substitute this into the first formula of \eqref{F},
then
\[F(2s;1) = \exp\left[T_1(1)s-\frac12M(1)+\mathcal{O}(\e^{-C|s|})\right].\]
Furthermore, using the third of formulas \eqref{F}, asymptotics \eqref{q1}, \eqref{intq1}, and expression \eqref{qu} of $u=q^2-q_x,$ integrating by parts,
we find that 
\[F(2s;1) = \exp\left[T_1(1)s-\frac12K(1)+\mathcal{O}(\e^{-C|s|})\right].\]

Hence, $K(1)=M(1).$
%
%
\end{proof}

\begin{rem}\label{remAsympKdV}
The quantity $K(1)$ from the
formula in Corollary \ref{cor1} was found by non-rigorous computations by Forrester \cite[(2.26), (2.30)]{F},
in the form of a slowly convergent series,
\[K(1)=-2\(\ln2-\frac14+\frac{1}{4\pi}\sum\limits_{n=2}^{\infty}\frac{1}{n}\(-\pi+\sum\limits_{j=1}^{n-1}\frac{1}{\sqrt{j(n-j)}}\)\)
\approx
-0.1254.
\]
On the other hand, Baik and Bothner \cite[unnumbered formula for $\eta_0(1) = \e^{-\frac12K(1)}$ on p.6, formula (1.16)]{BB18}
found numerically another value of $K(1),$
\[K(1) \approx 0.56798925.\]

The fact, that $K(1)$ is a conserved quantity of the KdV, allows, in principle, to compute $K(1)$ by using (known) large time $t\to+\infty$ asymptotics of the $u(x,t;1).$ 
Indeed, for $t=0$ we might study only the asymptotics $x\to\pm\infty$ of $u(x,0;\gamma).$ When $t\to+\infty,$ we know in principle the asymptotics for $u(x,t;\gamma)$ for all $x,$ which means that we can find an expression for integral of $u(x,t;\gamma).$ Easier said than done, and we do not pursue this issue here. 
For a note, we list the known leading asymptotic as $t\to+\infty$ terms for $u(x,t;\gamma)$ 
(\cite{DVZ}
, \cite[Thm 5.4]{GT}),

\begin{enumerate}
\item $x<-\varepsilon t$ (similarity asymptotics):
\begin{equation}\label{asympKdV}
\begin{split}
&u(x,t)\sim \sqrt{\frac{4\nu(\xi)k_0(\xi)}{3t}} \, \sin\(16tk_0^3(\xi) -\nu(\xi)\ln\(192tk_0^3(\xi)+\delta(\xi)\) \),
\end{split}
\end{equation}
\[\begin{split}
&\int\limits_{x}^{+\infty}u(\tilde x,t)\d \tilde x
\sim 
\frac{-1}{\pi}\int\limits_{-k_0}^{k_0}\ln\(1-|R(z)|^2\)\d z
-\sqrt{\frac{\nu(\xi)}{3k_0 t}} \, \cos\(16tk_0^3 -\nu(\xi)\ln\(192tk_0^3 +\delta(\xi)\) \),
\end{split}
\]
\\where\\
$\xi=\frac{x}{12t},\ k_0=k_0(\xi)=\sqrt{-\xi}, \ \nu(\xi) = \frac{-1}{2\pi}\ln\(1-|R(k_0(\xi))|^2\),$
\\
$\delta(\xi) = \frac{\pi}{4} - \arg R(k_0(\xi))+\arg\Gamma(\i\nu(k_0(\xi)))-\frac{1}{\pi}\int\limits_{-k_0}^{k_0}\ln\(\frac{1-|R(\zeta)|^2}{1-|R(k_0)|^2}\)\frac{\d\zeta}{\zeta-k_0}.$

\item 
$-C<\frac{x}{t^{1/3}}<C:$
\[
u(x,t)\sim\frac{1}{(3t)^{2/3}}\(p^2(s)-p'(s)\),\quad \mbox{ where } s=\frac{x}{(3t)^{1/3}},
\]
and 
$p(s)$ is the solution of the Painlev\'e II equation
\[p''(s) - sp(s) - 2p^3(s) = 0,\] fixed by its asymptotics
$p(s) \sim - R(0)Ai(s),\ s\to+\infty.$
For $R(0)>-1, $ $p(s)$ is oscillating and vanishing as $s\to-\infty,$ and for $R(0)=-1,$ $p(s)\sim\sqrt{-\frac12s}$ as $s\to-\infty$ (see Hastings, McLeod \cite{HM} for details).

\item in the case $R(0)=-1,$ there is an additional region 
$-C_2<\frac{x}{t^{1/3}(\ln t)^{2/3}}<-C_1:$
with an elliptic asymptotics,
\[
u(x,t)\sim\frac{-2x}{3t}\left[
A(\alpha) + B(\alpha) \mathrm{cn}^2\(2K(\alpha)\theta+\theta_0;\alpha\),\
\right],
\]
where the slow parameter $\alpha=\alpha(s)$ is determined by 
\[\alpha = \alpha(s) = 1-\frac{a^2(s)}{b^2(s)},\quad \mbox{where } 0\leq a(s)\leq b(s)\leq\sqrt{2} \mbox{ are determined by the system} \]
\[ 
a^2+b^2 = 2, s=24\int_a^b\sqrt{(y^2-a^2)(b^2-y^2)}\d y, 0\leq s\leq 8^{3/2}
\]
and we refer the reader to the original paper \cite{DVZ} for details about the other quantities in the above formula.

\item
$x>\varepsilon t:$ $\quad u(x,t)\sim 0$ (there are no solitons in our case).
\end{enumerate}

Here $\varepsilon, C, C_1, C_2$ are generic positive constants. Between the regions there are gaps, which to the best of our knowledge were not studied in the literature.

\end{rem}

\begin{Crem}
Substituting the above asymptotics of $u$ into the expression \eqref{K} of $K(\gamma)$, and making some heuristic computations

(like those: since $\int_{-\infty}^{-\varepsilon t}u(x,t)\d x = \mathcal{O}(1),$ then $\int_{-\infty}^{-\varepsilon t}x u(x,t)\d x = \mathcal{O}(t), t\to+\infty;$ 

furthermore, $\int_{-C t^{1/3}}^{C t^{1/3}}x
\frac{1}{(3t)^{2/3}}\(p^2\(\frac{x}{(3t)^{1/3}}\) - p'\(\frac{x}{(3t)^{1/3}}\)\)\d x = \mathcal{O}(1), t\to+\infty$),

\noindent we guess that the similarity asymptotics give the contribution of the order $t^{1}$ into $K(\gamma),$ and Painleve asymptotics give a contribution of the order $t^0.$ Let us mention, that the contribution of the order $t^{1}$ is always non-zero, even when there are no solitons, as in our case. Furthermore, the integral $\int_{-\infty}^{+\infty} s(p^2(s)-p'(s))\d s,$ which might be convergent for $\gamma<1,$ but definitely divergent for $\gamma = 1,$ might be regularized for $\gamma = 1.$

Indeed, function $p(s),$ corresponding to the case $\gamma=1,$
has the asymptotics as $s\to\pm\infty:$
\[
p(s) \sim Ai(s),\ s\to+\infty,
\quad p(s) = \frac{\sqrt{-s}}{\sqrt{2}}\(1+\frac{1}{8 s^3}-\frac{73}{128s^6}+\mathcal{O}(s^{-9})\),\ s\to-\infty,
\]
which admits element-wise differentiation w.r.t. $s,$
so that
\[s(p^2(s)-p'(s)) = \frac12s^2-\frac{\sqrt{-s}}{2\sqrt{2}}+\frac18s^{-1}+\mathcal{O}(|s|^{-5/2}), \ s\to-\infty.\]
We have a convergent integral
\[P=
\int\limits_{-\infty}^s\left[\tilde s(p^2(\tilde s)-p'(\tilde s)) - \frac12\tilde s^2+\frac{\sqrt{-\tilde s}}{2\sqrt{2}}-\frac{1}{8\tilde s}\right] \d\tilde s+
\int\limits_{s}^{+\infty}
\tilde s\(p^2(\tilde s)-p'(\tilde s)\)\d\tilde s,\]
which does not depend on the choice of $s<0.$

We would expect that $K(1)$ from Corollary \ref{cor1} is related to $P$,
$K(1)\asymp P.$

\end{Crem}

\begin{rem}
A more practical way to compute $K(\gamma),$ $gamma\in[0,1],$ numerically is to do this at the time $t=0,$ by using  
the main integral equations of the inverse scattering problem (a.k.a Marchenko equations, Gelfand-Levitan-Marchenko equations) 
\cite[formulas (3.5.18), (3.5.18'), (3.5.21), 
p.290]{Marchenko}, which are Volterra integral equations.
For the spectral problem $-\partial_x^2\psi + u(x,0;\gamma)\psi = k^2\psi$ they are 
\[\begin{split}
\mathcal{K}_+(x,y) + \mathcal{R}_+(x+y)+\int\limits_{x}^{+\infty}\mathcal{K}_+(x,z)\mathcal{R}_+(z+y)\d z = 0,\quad y\geq x,\\
\mathcal{K}_-(x,y) + \mathcal{R}_-(x+y)+\int\limits^{x}_{-\infty}\mathcal{K}_-(x,z)\mathcal{R}_-(z+y)\d z = 0,\quad y\leq x,
\end{split}\]
where 
\[\mathcal{R}_+(x) = \frac{1}{2\pi}\int\limits_{-\infty}^{+\infty}R(k)\e^{\i k x}\d k=-\sqrt{\frac{\gamma}{\pi}}\e^{-x^2},
\qquad
\mathcal{R}_-(x) = \frac{1}{2\pi}\int\limits_{-\infty}^{+\infty}L(k)\e^{-\i k x}\d k,\]
and $R(k)=R(k;\gamma)$ is defined in \eqref{R}, and $L(k)=L(k;\gamma)$ is defined in \eqref{L}.
The link with the function $u(x)=u(x;\gamma)=u(x,0;\gamma)$ is given by the formulas 
\[\mathcal{K}_+(x,x) =\frac12 \int\limits_{x}^{+\infty} u (\tilde x)\d\tilde x,
\quad
\mathcal{K}_-(x,x) = \frac12\int\limits^{x}_{-\infty} u (\tilde x)\d\tilde x.\]
Then the integral \eqref{K}, computed at $t=0,$ can be treated as follows: for any real $x_0,$
\[\begin{split}K(\gamma) &= \int\limits_{-\infty}^{+\infty}x u(x,0;\gamma)\d x =
\int\limits_{-\infty}^{x_0}x u(x,0;\gamma)\d x + \int\limits_{x_0}^{+\infty}x u(x,0;\gamma)\d x = 
\\
&=
-\int\limits_{-\infty}^{x_0}(x_0-x) u(x,0;\gamma)\d x + \int\limits_{x_0}^{+\infty}(x-x_0) u(x,0;\gamma)\d x 
+x_0 \int\limits_{-\infty}^{+\infty}u(x)\d x
=
\\
&=
-\int\limits_{-\infty}^{x_0}\(\int_{-\infty}^{x}u(\tilde x,0;\gamma)\d\tilde x\)\d x +
\int\limits^{+\infty}_{x_0}\(\int^{+\infty}_{x}u(\tilde x,0;\gamma)\d\tilde x\)\d x
+x_0 \int\limits_{-\infty}^{+\infty}u(x,0;\gamma)\d x\\
&=
-\int\limits_{-\infty}^{x_0}2\mathcal{K}_-(x,x)\d x +\int\limits^{+\infty}_{x_0}2\mathcal{K}_+(x,x)\d x+
x_0 \cdot \sqrt{\frac{2}{\pi}}\;\mathrm{Li}_{\frac32}(\gamma),
\end{split}
 \]
 since 
 $\int\limits_{-\infty}^{+\infty}u(x,0;\gamma)\d x = 2T_1(\gamma) = \sqrt{\frac{2}{\pi}}\;\mathrm{Li}_{\frac32}(\gamma)$
 in view of formulas \eqref{qu} and asymptotics \eqref{qgamma}, \eqref{q1}.
\end{rem}

\begin{rem}
Function $q(x;\gamma)$ seems to be positive.
For $x\to-\infty,$ approximately
\[
\begin{split}
&
\int\limits_{x}^{+\infty} q^2(z;\gamma)\d z-2T_1(\gamma) \approx -q(x;\gamma) \(\frac{\e^{2x\kappa_{\gamma}}L_{-1}(\gamma)}{2\kappa_{\gamma}}\),\quad \gamma<1,
\\
&
\int\limits_{x}^{+\infty} q^2(z;1)\d z-2T_1(1) \approx -q(x;1).
\end{split}
\]
\end{rem}

\begin{rem}
Consider rarefaction problem for KdV,
$u\to c^2, x\to-\infty,\quad u\to 0, x\to+\infty.$
It has conserved quantities, (independent of $x$ and $t$)
\[
\begin{split}
&K = 3\int\limits_{-\infty}^x(u^2-c^4)\d\tilde x+3\int\limits_{x}^{+\infty}u^2\d\tilde x+3c^4 \(x+4c^2t\),
\\
&
H=\int\limits_{-\infty}^{x}z\(u(z,t)-c^2\)\d z
+\int\limits_{x}^{+\infty}zu(z,t)\d z+\frac{c^2}{2}x^2-6c^6t^2+Kt.
\end{split}\]
\end{rem}

\begin{Crem} Numerical experiment (based on section \ref{sect_num}) allows us to conjecture that
\begin{equation}\label{Lkappa}
\frac{L_{-1}(\gamma)}{2\kappa_{\gamma}} = 1 -L_1(1)\kappa_{\gamma}+l_2\kappa_{\gamma}^2 -l_3\kappa_{\gamma}^3+ \mathcal{O}(\kappa_{\gamma}^4),\quad \gamma\to1-0.
\end{equation}

Then formulas \eqref{q1} might be obtained from formulas \eqref{qgamma}
by taking formal limit $\kappa_{\gamma}\to0$, and neglecting terms of positive order in $\kappa_{\gamma}.$

Indeed, we get formally that for $\gamma\to 1-0,$
\[
q(x;\gamma)\sim 
\frac{2}{L_1(1)-2x}
+
\frac{(2l_2-L_1(1)^2)\kappa_{\gamma}}{(L_1(1)-2x)^2}
+\mathcal{O}\(\kappa_{\gamma}^3\).
\]
Numerics $l_2\approx 0.678\,838\,896\,877 \approx \frac12 L_1(1)^2$ suggest us to conjecture  $l_2 = \frac12 L_1(1)^2,$ 
and then we can simplify the expression for $\kappa_{\gamma}^2$  term:
\[
q(x;\gamma)
\sim
\frac{2}{L_1(1)-2x}
+
\frac{6l_3+2x(4x^2-6L_1(1) x+3L_1(1)^2)}
{3(L_1(1)-2x)^2}
\kappa_{\gamma}^2
+\mathcal{O}\(\kappa_{\gamma}^4\).
\]
Sweet life ends here: because of presence of $x,$ we can not make this term to be equal to $0.$
We have 
\[l_3 \approx 0.236\,014\,8731\ \neq \ \frac16L_1(1)^3\approx 0.263\,659\,741.\]
\end{Crem}

\begin{Crem}
It seems that $M(\gamma)\to M(1)$ as $\gamma\to1-0.$
Indeed, splitting the integral for $M(\gamma),$ $\gamma<1$ into two parts $(-\infty,x)$ and $(x,+\infty)$ for $x$ sufficiently large negative, and substituting asymptotics \eqref{qgamma}, we find
\[\gamma<1: 
\int\limits_{-\infty}^x
\tilde x q^2(\tilde x;\gamma)\d\tilde x\sim
\frac{4x\kappa \frac{L_{-1}(\gamma)^2\e^{4x\kappa}}{4\kappa_{\gamma}^2}}
{1-\frac{L_{-1}(\gamma)^2\e^{4x\kappa}}{4\kappa_{\gamma}^2}}
+\ln\(1-\frac{L_{-1}(\gamma)^2\e^{4x\kappa}}{4\kappa_{\gamma}^2}\),
\]
which in the $\kappa_{\gamma}\to0$ limit gives, using \eqref{Lkappa},
\[
\frac{2x}{L_1(1)-2x}+\ln(L_1(1)-2x)+\ln(2\kappa_{\gamma})+\mathcal{O}(\kappa_{\gamma}).
\]
On the other hand,
\[
\gamma=1:\int\limits_{-\infty}^x\(\tilde xq^2(\tilde x;1)-\frac{1}{\tilde x}\)\d \tilde x
\sim
\frac{2x}{L_1(1)-2x}+\ln(L_1(1)-2x)+1-\ln2-\ln|x|.
\]
Comparing, we come to a formal conclusion that $M(1-0)=M(1).$
\end{Crem}

\textbf{Acknowledgments.}
A.M. acknowledges the support of the H2020-MSCA-RISE-2017 PROJECT No. 778010 IPADEGAN, and the support of the organizers of 
\href{https://www.chairejeanmorlet.com/2019-1-grava-bufetov-2104.html} {the conference }
`Integrability and Randomness in Mathematical Physics and Geometry', April 8--12, 2019, CIRM (Marseille, Luminy, France), and Thomas Bothner, from where he learned about the problem.
Also A.M. would like to thank \href{https://www.uva.nl/en/profile/g/a/o.gamayun/o.gamayun.html}{Oleksandr Gamayun}, who pointed out that the quntity \eqref{intq1} can be found from the solution of the RHP by expanding it at the origin, and to \href{http://www.math.sissa.it/users/pieter-roffelsen}{Pieter Roffelsen} for useful remarks.

\section{Proof of (a),(b)}
\begin{lemma}
\begin{enumerate}
\item For any fixed $\gamma\in[0,1],$ $x\in\mathbb{R},$
the Riemann-Hilbert problem \ref{RightRHP} has the unique solution. This solution is continuous in parameters $(x;\gamma)\in\mathbb{R}\times[0,1].$
\item For any $\gamma\in[0,1]$ and $x\in\mathbb{R},$ the solution of the RHP \ref{RightRHP} is infinitely differentiable in $x.$
\end{enumerate}
\end{lemma}
\begin{proof}
The proof is almost word-to-word repetition of the similar proof from \cite{KM19}, p. 13-17 (for the existence part also cf \cite{BB18}).
For the convenience of the reader we present it also here.

\textbf{Existence.}
Let $x\in\mathbb{R}$ and $\gamma\in[0,1]$ be fixed. We look for the solution $\M(x;k;\gamma)$ of the RHP \ref{RightRHP} in the form:
\begin{equation}\label{M_int}\M(x;k;\gamma) = \1+\frac{1}{2\pi\i}\int\limits_{\mathbb{R}}\frac{[\1+\Z(x;s;\gamma)][\1-\J_M(x;s;\gamma)]\d s}{s-k},\quad s\in\mathbb{C}\setminus\mathbb{R}.\end{equation}
One can show that the Cauchy integral \eqref{M_int} satisfies all the properties of the RHP if and only if the matrix $\Z(x;k;\gamma)$ satisfies the singular integral equation
\begin{equation}\label{SIE_Z}
\Z(x;s;\gamma) - \mathcal{K}[\Z](x;s;\gamma) = \mathbf{F}(x;s;\gamma),
\quad s\in\mathbb{R}.
\end{equation}
The singular integral operator $\mathcal{K}$ and the right-hand side $F(x;k;\gamma)$ are as follows:
\[\mathcal{K}[\Z](x;s;\gamma) = \frac{1}{2\pi\i}\int\limits_{\mathbb{R}}\frac{\Z(x;z;\gamma)[\1-\J_M(x;z;\gamma)]}{(z-s)_+}\d z,\]
\[
\mathbf{F}(x;s;\gamma) = \frac{1}{2\pi\i}\int\limits_{\mathbb{R}}\frac{\1-\J_M(x;s;\gamma)}{(z-s)_+} \d z.
\]
We consider this integral equation in the space $L^2(\mathbb{R})$ of $2\times2$ matrix complex-valued functions $\Z(k) := \Z(x;k;\gamma).$ The operator $\mathcal{K}$ is defined by the jump matrix $\J_M(x;k;\gamma)$ and the generalized function $\dfrac{1}{(z-s)_+} = \lim\limits_{k\to s, k\in+\mbox{side}}\dfrac{1}{z-k}.$

It is a classical fact that the Cauchy operator \[C_+[f](s) = \frac{1}{2\pi\i}\int\limits_{\mathbb{R}}\frac{f(z)}{(z-s)_+}\d z\] is bounded in the space $L^2(\mathbb{R}).$

The matrix-valued function $\1-\J_M(x;k;\gamma)$ as a function of variable $k$ is in the space $L^2(\mathbb{R}).$ Hence, the function $\mathbf{F}(x;k;\gamma)$ is also in $L^2(\mathbb{R}).$ The matrix-valued function $\1-\J_M(x;k;\gamma)$ is bounded as a function of variable $k:$
$\1-\J_M(x;k;\gamma)\in L^{\infty}(\mathbb{R}).$ Thus, operator $Id-\mathcal{K}$ is an operator acting in $L^2(\mathbb{R})$ ($Id$ is the identical operator). The contour $\mathbb{R}$ and the jump matrix $\J_M(x;k;\gamma)$ satisfy the Schwartz reflection principle \cite[Theorem 9.3]{Zhou89}:
\begin{itemize}
\item the contour $\mathbb{R}$ is symmetric with respect to the real axis $\mathbb{R};$
\item on parts of the contour outside of the real line, with orientation respecting the symmetry w.r.t. $\mathbb{R},$ we have 
$\J_M(x;k;\gamma) = \ol{\J_M(x;\ol k;\gamma)}^T;$
\item the jump matrix $\J_M(x;k;\gamma)$ has a positive definite real part for $k\in \mathbb{R}.$
\end{itemize}

In our case the contour coincides with the real axis, and hence the second condition of the Schwartz reflection principle is trivial in our case.

Then Theorem 9.3 from \cite{Zhou89} (p.984) guarantees the $L^2$ invertibility of the operator $Id -\mathcal{K}.$
Therefore, the singular integral equation \eqref{SIE_Z} has a unique solution $\Z(x;k;\gamma)\in L^2(\mathbb{R})$ for any fixed $x\in\mathbb{R},$ $\gamma\in[0,1]$ and formula \eqref{M_int} gives the solution of the above RHP.

The operator $Id-\mathcal{K}$ depends continuously on the parameters $(x;\gamma)\in\mathbb{R}\times[0,1].$ Therefore the inverse operator $(Id-\mathcal{K})^{-1}$  also has this property. Hence, the solution $\Z(x;k;\gamma)$ of the singular integral equation \eqref{SIE_Z} also depends continuously on $x,\gamma.$ From representation \eqref{M_int} we obtain the required statement for $\M(x;k;\gamma).$

\textbf{Uniqueness.} The uniqueness for the RHP \ref{RightRHP} in the space $L^2(\mathbb{R})$ is proved in \cite{Deift} (p.194-198).

\textbf{Smoothness.}
We can differentiate the singular integral equation \ref{SIE_Z} in $x$ as many times as desired. Indeed, to differentiate this equation and matrix $\Z$ it is sufficient that its formal derivatives are convergent. The function $\1-\J_M(x;s;gamma)$ is responsible for decaying of integrands in the singular integral equation. Since on the real line $\1-\J_M(x;s;\gamma)$ decays exponentially fast w.r.t. $s\to\pm\infty.$ 
Singular integral equations obtained from \eqref{SIE_Z} by differentiation w.r.t. $x$ are of the same form as the original one \eqref{SIE_Z}, ony the r.h.s. of these equations vary.
Indeed, writing \eqref{SIE_Z} in the form
\[\Z(x;s;\gamma) - C_+[\Z(x;s;\gamma)(\1-\J_M(x;s;\gamma))] = \mathbf{F}(x;s;\gamma),\]
for its formal derivative w.r.t. $x$ we get
\[\Z_x(x;s;\gamma) - C_+[\Z_x(x;s;\gamma)(\1-\J_M(x;s;\gamma))] = \mathbf{F}_1(x;s;\gamma):=\mathbf{F}_x(x;s;\gamma)-
C_+[\Z(x;s;\gamma)\J_{M,x}(x;s;\gamma)],\]
and so on for higher derivatives.

Since the left-hand-side operator is the same as in \eqref{SIE_Z}, it is invertible, and this provides a unique solvability and existence of the partial derivatives of $\Z(x;s;\gamma)$ with respect to $x.$ Hence, the same is true for $\Z(x;k;\gamma).$
\end{proof}

\begin{lemma}\label{lem_sym}
The solution $\M(x;k;\gamma)$ of the RHP \ref{RightRHP} has the symmetries
\[
\begin{split}
&\M(x;k;\gamma)=
\begin{pmatrix}0 & 1 \\ 1 & 0\end{pmatrix}
\ol{\M(x;\ol k;\gamma)}
\begin{pmatrix}0 & 1 \\ 1 & 0\end{pmatrix}
,
\\
&\M(x;k;\gamma)=
\begin{pmatrix}0 & 1 \\ 1 & 0\end{pmatrix}
\M(x; -k;\gamma)
\begin{pmatrix}0 & 1 \\ 1 & 0\end{pmatrix}.
\end{split}\]
\end{lemma}
\begin{proof}
This  follows from the corresponding symmetries for the jump matrix,
\[
\begin{split}
&\J_M(x;k;\gamma)^{-1}=
\begin{pmatrix}0 & 1 \\ 1 & 0\end{pmatrix}
\ol{\J_M(x;\ol k;\gamma)}
\begin{pmatrix}0 & 1 \\ 1 & 0\end{pmatrix}
,
\\
&\J_M(x;k;\gamma)^{-1}=
\begin{pmatrix}0 & 1 \\ 1 & 0\end{pmatrix}
\J_M(x; -k;\gamma)
\begin{pmatrix}0 & 1 \\ 1 & 0\end{pmatrix}.
\end{split}\]
\end{proof}

\begin{lemma}
Let $\mathbf{\Phi}(x;k;\gamma) = \M(x;k;\gamma)\e^{\i k x\sigma_3}.$
Then $\mathbf{\Phi}(x;k;\gamma)$ satisfies the equation
\[\mathbf{\Phi}_x(x;k;\gamma) + \i k  \s_3 \mathbf{\Phi}(x;k;\gamma) = \mathbf{Q}(x;\gamma)\mathbf{\Phi}(x;k;\gamma),\quad x\in\mathbb{R},\ \gamma\in[0,1],\]
where \[\mathbf{Q}(x;\gamma) = \begin{pmatrix} 0 &  -q(x;\gamma) \\  -q(x;\gamma) & 0\end{pmatrix},\]
with the function $q(x;\gamma)$ given by
\[q(x;\gamma) = -2\i \lim\limits_{k\to\infty}
k[\M(x;k;\gamma)]_{12} = 
\frac{-1}{\pi}\int\limits_{\mathbb{R}}
\([\1+\M(x;s;\gamma)][\J_M(x;s;\gamma)-\1]\)_{12}\d s.
\]
\end{lemma}
\begin{proof}
The proof is standard and uses the fact that the derivative $\mathbf{\Phi}_x$ has the same jump condition as $\mathbf{\Phi},$ and then relies on the Liouville theorem applied to $\mathbf{\Phi}_x\mathbf{\Phi}^{-1},$ and we will suppress it. The realness of $q(x;\gamma)$ follows from symmetries from Lemma \ref{lem_sym}.
\end{proof}

\begin{cor}
For any fixed $\gamma\in[0,1],$ the function 
$q(x;\gamma)$ is smooth for $x\in\mathbb{R}.$
\end{cor}

\begin{lemma}\label{lem_expansion}
For any $n\in\mathbb{N},$ the solution $\M(x;k;\gamma)$ of the RHP \ref{RightRHP} has the expansion as $k\to\infty$
\[\M(x;k;\gamma) = \1 +\sum_{j=1}^{n}\( \frac{1}{k^{2j-1}}\begin{pmatrix}
\i A_{2j-1} & i B_{2j-1} \\ -\i B_{2j-1} & -\i A_{2j-1}
\end{pmatrix}
+
\frac{1}{k^{2j}}\begin{pmatrix}
A_{2j} &  B_{2j} \\  B_{2j} & A_{2j}
\end{pmatrix}\)+\mathcal{O}(k^{-2n-1}),
\]
where all $A_j = A_j(x;\gamma),\ B_j = B_j(x;\gamma)$ are real.
Furthermore,
\[\begin{split}
&
q(x;\gamma) = 2 B_1(x;\gamma) \in \mathbb{R},
\\
&
\partial_x( A_{1}(x;\gamma)) = \frac{1}{2} q^2(x;\gamma) = 2 B_1(x;\gamma)^2\in\mathbb{R}.
\end{split}\]
\end{lemma}
\begin{proof}
The possibility to expand the function $\M(x;k;\gamma)$ for large $k$ follows from the representation \eqref{M_int} and the fact that the $\1-\J_M(x;s;\gamma)$ is exponentially small for $s$ on the infinite part of the real line.
The symmetries of the elements of the expansion follows from the symmetries in Lemma \ref{lem_sym}.
Finally,
writing an expansion
\[\M(x;k;\gamma) = \1 + \frac{\mathbf{m}_1(x;\gamma)}{k} + \frac{\mathbf{m}_2(x;\gamma)}{k^2}+\ldots,\] and substituting this into 
\[{\M}_x + \i k [{\s}_3, \M] = \mathbf{Q} \M,\]
where $[\mathbf{A},\mathbf{B}]=\mathbf{A}\mathbf{B}-\mathbf{B}\mathbf{A}$ is the matrix commutator, we obtain 
\[\frac{\mathbf{m}_{1,x}}{k} + \frac{\mathbf{m}_{2,x}}{k^2}+\ldots +
\i[\s_3,\mathbf{m}_1]+\frac{\i[\s_3,\mathbf{m}_2]}{k} 
+
\frac{\i[\s_3, \mathbf{m}_3]}{k^2}+\ldots
=
\mathbf{Q}+\frac{\mathbf{Q}\mathbf{m}_1}{k}+\frac{\mathbf{Q}\mathbf{m}_2}{k^2}+\ldots.
\]
Comparing the (off-diagonal) terms of the order $k^0,$ and diagonal terms of the order $k^{-1},$
we find that
\[
\begin{split}
q(x;\gamma) = -2\i (\mathbf{m}_1)_{12},
\qquad
\partial_x (\mathbf{m}_1)_{11} = \frac{\i}{2} q^2,
\end{split}\]
which finishes the proof.
\end{proof}

\section{Analysis for $x\to+\infty.$}
\begin{lemma}
Let $A_1(x;\gamma)$ be as in Lemma \ref{lem_expansion}. Then
\[A_1(x;\gamma) = -\frac12\int\limits_{x}^{+\infty}q^2(\tilde x;\gamma) \d \tilde x = -
2\int\limits_{x}^{+\infty}B_1(\tilde x;\gamma)^2\d \tilde x.\]

\end{lemma}
\begin{proof}
Let us draw two lines $L_1 = \mathbb{R} + \i, L_2 = \mathbb{R} - \i,$ with orientation as on the real line. Denote the domain between $L_1$ and $\mathbb{R}$ by $\Omega_1,$ the other domain in $\Im k>0$ by $\Omega_3,$ the domain between $L_2$ and $\mathbb{R}$ by $\Omega_2,$ and the remaining domain by $\Omega_4.$
Denote $\Sigma = L_1\cup L_2$ to be an oriented contour.

Define a function
\[
\P(x;k;\gamma) = \M(x;k;\gamma)\cdot
\begin{cases}
\begin{pmatrix}1 & 0 \\ -R(k;\gamma) \e^{2\i k x} & 1\end{pmatrix},
k\in\Omega_1,
\\
\begin{pmatrix}1 & -\ol{R(\ol k; \gamma)} \e^{-2\i k x} \\ 0 & 1\end{pmatrix},
k\in\Omega_2,
\\
\1, \mbox{elsewhere}.
\end{cases}
\]
The function $P(x;k;\gamma)$ solves the following RHP.
\begin{RHP}\label{PosxRHP}
To find a $2\times2$ matrix-valued function $\P(x,t;k;\gamma)$ that satisfies the following properties:
\begin{itemize}
\item analyticity: $\P(x;k;\gamma)$ is analytic in $k\in\(\mathbb{C}\setminus\Sigma\),$\\
and continuous up to the boundary $k\in\Sigma;$
\item jumps:
 $\P_- = \P_+ \J_{P},$ where
\[\J_P=\begin{pmatrix}1 & 0 \\ -R(k;\gamma)\e^{2\i kx} & 1 \end{pmatrix}, k\in L_1,
\qquad
\J_P = \begin{pmatrix}1 & \ol{R(\ol k;\gamma)}\cdot\e^{-2\i kx} \\ 0 & 1 \end{pmatrix},
 \ k\in L_2;\]
\item asymptotics at the infinity:
\[\P(k)\to \1\quad \mbox{ as } \quad k\to\infty.\]
\end{itemize}
\end{RHP}

For $x\to+\infty,$ the jump matrix for $\P$ is uniformly exponentially close to $\1$ everywhere on the contour $\Sigma=L_1\cup L_2,$ and hence the matrix
\[\begin{pmatrix}\i A_1(x;\gamma) & \i B_1(x;\gamma) \\ -\i B_1(x;\gamma) & -\i A_1(x;\gamma)\end{pmatrix} = \lim\limits_{k\to\i\infty}k(\M(x;k;\gamma)-I) = 
\lim\limits_{k\to\i\infty}k(\P(x;k;\gamma)-I)\]
tends to $0$ exponentially fast as $x\to+\infty.$
Then, firstly, $\int_{x}^{+\infty}B_1^2(\tilde x;\gamma)$ exists (converges), and secondly, by Lemma \ref{lem_expansion}, from 
\[\partial_x A_1(x;\gamma) = 2B_1(x;\gamma)^2\]
we get 
\[A_1(x;\gamma) = -2\int\limits_{x}^{+\infty}B_1(\tilde x;\gamma)^2\d\tilde x.\]

\end{proof}

\section{Analysis for $x\to-\infty$ and proof of (c), (d)}

\subsection{Functions $T(k;\gamma), L(k;\gamma), \delta(k;\gamma).$}\label{sect_num}
First of all, let us introduce some auxiliary functions.

\noindent 
Define an entire function $T(k;\gamma)$ by the formula
\[
T(k;\gamma) = \begin{cases}
\exp\left[\dfrac{1}{2\pi\i}\displaystyle\int\limits_{-\infty}^{+\infty}
\frac{\ln(1-|R(s;\gamma)|^2)\ \d s}{s-k}\right],\quad \Im k>0,
\\
(1-R(k;\gamma)\ol{R(\ol k;\gamma)})\cdot\exp\left[\displaystyle\frac{1}{2\pi\i}\int\limits_{-\infty}^{+\infty}
\frac{\ln(1-|R(s;\gamma)|^2)\ \d s}{s-k}\right],\quad \Im k<0.
\end{cases}\]
Furthermore, define {\it the left reflection coefficient} $L(k;\gamma)$ by the formula 
\begin{equation}\label{L}L(k;\gamma) = -\frac{\ol{R(\ol k ;\gamma)}\ T(k;\gamma)}{\ol{T(\ol k ;\gamma)}}  = \frac{-\ol {R(\ol k ;\gamma)}\ T(k;\gamma)^2}{1-R(k;\gamma)\ol{R(\ol k;\gamma)}} 
\end{equation}

We collect the properties of $T(k;\gamma), L(k;\gamma)$ in the following lemma.
\begin{lemma}
\begin{enumerate}
\item $T(k;\gamma), R(k;\gamma)$ are entire function in $k,$
$L(k;\gamma)$ is a meromorphic function with a pole at $k = \i\kappa_{\gamma}.$
\item Zeros and poles.
\begin{enumerate}\item The only zero of $T(k;\gamma)$ is a simple pole at $k=-\i\kappa_{\gamma},$
where \[\kappa_{\gamma} = \sqrt{-2\ln\gamma}\geq 0.\]
It is a simple zero in both cases $\gamma<1$ and $\gamma = 1.$

\item For $\gamma\in(0,1),$ the function $L(k;\gamma)$ has a simple pole at $k = \i \kappa_{\gamma},$ and a simple zero at $k = -\i\kappa_{\gamma}.$

For $\gamma = 1,$ the function $L(k;1)$ is an entire function, and does not vanish at $k=0.$
\end{enumerate}

\item Scattering relations:
\[
\begin{split}
&
T(k;\gamma)\, \ol{T(\ol k;\gamma)}=1-R(k;\gamma)\, \ol{R(\ol k;\gamma)} = 1-L(k;\gamma)\, \ol{L(\ol k;\gamma)},
\quad k\in\mathbb{C}.
\\
&
T(k;\gamma)\, \ol{T(\ol k;\gamma)} = 1-L(k;\gamma)\, \ol{L(\ol k;\gamma)},\quad  k\in\mathbb{C}.
\\
& 
T(k;\gamma)\ol{R(\ol k;\gamma)} + L(k;\gamma)\ol{T(\ol k;\gamma)}=0.
\end{split}
\]

\item Symmetries:
\[\ol{T(-\ol k;\gamma)} = T(k;\gamma),\qquad 
\ol{L(-\ol k;\gamma)} = L(k;\gamma),\qquad \ol{R(-\ol k;\gamma)} = R(k;\gamma).\]

\item Large $k$ asymptotics of $T(k;\gamma)$ in $\Im k\geq 0:$
\[T(k;\gamma) = 1-\frac{\i\, T_1(\gamma)}{k}+\mathcal{O}(k^{-2}),\quad k\to\infty,\Im k\geq 0,\]
where 
$T_1(\gamma) = \frac{-1}{2\pi}\int\limits_{-\infty}^{+\infty}\ln(1-\gamma\e^{-\frac{s^2}{2}})\d s
=\frac{1}{\sqrt{2\pi}}
Li_{\frac32}(\gamma)>0.$

\item Pole condition of $L(k;\gamma)$ at $k = \i\kappa_{\gamma}:$
\[
L(k;\gamma)
=-\(
\frac{\i L_{-1}(\gamma)}{k-\i\kappa_{\gamma}}
+L_0(\gamma)
+\i L_1(\gamma) (k-\i\kappa_{\gamma})
+L_2(\gamma) (k-\i\kappa_{\gamma})^2+\ldots\),\quad k\to\i\kappa_{\gamma}
\]
For $\gamma = 1,$
\begin{equation}\label{L2L1}L(k;1) = -\(1 + \i L_1(1)k+L_2(1)k^2+\mathcal{O}(k^3)\),\qquad L_2(1)=-\frac12L_1(1)^2-\frac14,
\end{equation}
and all $L_{j}(\gamma)$ are real.

Furthermore,
\[
\frac{L_{-1}(\gamma)}{2\kappa_{\gamma}}
=1 - L_1(1)\kappa_{\gamma}+\frac12 L_1(1)^2 \kappa_{\gamma}^2+\mathcal{O}(\kappa_{\gamma}^2),\quad \gamma\to1-0.
\]
\end{enumerate}
\end{lemma}

\begin{proof}
To prove that the function $T(k;\gamma)$ is indeed entire, it suffices to establish continuity across the real line. This follows by Sokhotsky-Plemelj formula. 
The {\it scattering relation} follows from the definition of $T(k;\gamma).$

Regarding poles and zeros,  observe first that for $\gamma\in(0,1],$ the only zeros of the function \[1-R(k)\ol{R(\ol k)}
=1-\gamma\e^{-\frac12 k^2}
=1-\e^{-\frac12 (k^2+\kappa_{\gamma}^2)}
\] are
$k=\pm\i\kappa_{\gamma}.$

This is sufficient to prove all the statements about zeros and poles for $\gamma\neq 1.$

To treat also $\gamma = 1,$ and being able to make transition for $\gamma\to1-0,$
it is useful to introduce two auxiliary functions, $\delta(k;\gamma)$ and $\widehat(k;\gamma;a).$

Namely, define 
\[\delta(k;\gamma) = \begin{cases}T(k;\gamma),\ \Im k>0,\\\ol{T(\ol k;\gamma)}^{-1}, \ \Im k<0.\end{cases}\]
The function $\delta(k;\gamma)$ solves the conjugation problem
\[
\frac{\delta_+(k;\gamma)}{\delta_-(k;\gamma)} = 1- |R(k;\gamma)|^2,\quad k\in\mathbb{R},
\]
and $\delta(k;\gamma)\to 1$ as $k\to\infty,$
Furthermore,
define 
\[\widehat\delta(k;\gamma;a) = 
\begin{cases}
\frac{k+\i a}{k + \i \kappa_{\gamma} }\delta(k;\gamma),\quad \Im k>0,
\\
\\
\frac{k - \i \kappa_{\gamma}}{k-\i a}\delta(k;\gamma),\quad \Im k<0.
\end{cases}\]
Here $a>\kappa_{\gamma}$ is an arbitrary parameter; we can keep $a = 1$ for all $\gamma\in(\frac{1}{\sqrt{\e}};1]$. The latter formula is valid also for $\gamma=1,$ when $\kappa_1 = 0.$
Function $\widehat\delta(k;\gamma)$ solves the following scalar conjugation problem:
\begin{equation}\label{dh}
\frac{\widehat\delta_+(k;\gamma)}{\widehat\delta_-(k;\gamma)} = \frac{k^2+a^2}{k^2+\kappa_{\gamma}^2}\(1- |R(k;\gamma)|^2\)
\equiv
(k^2+a^2)\frac{1-\e^{-\frac12\(k^2+\kappa_{\gamma}^2\)}}{k^2+\kappa_{\gamma}^2},\quad k\in\mathbb{R},
\end{equation}
and $\widehat\delta(k;\gamma)\to 1$ as $k\to\infty.$

The functions $\delta, \widehat\delta$ possesses the symmetries
\begin{equation}\label{delta_sym}
\quad \ol{\delta(-\ol k;\gamma)} = \delta(k;\gamma) = \dfrac{1}{\ol{\delta(\ol k;\gamma)}} = \dfrac{1}{\delta(-k;\gamma)},
\quad
\quad \ol{\widehat\delta(-\ol k;\gamma;a)} = \widehat\delta(k;\gamma;a) = \dfrac{1}{\ol{\widehat\delta(\ol k;\gamma;a)}} = \dfrac{1}{\widehat\delta(-k;\gamma;a)}.
\end{equation}
The function $\widehat\delta(k;\gamma;a)$ can be written explicitly,
\[
\widehat\delta(k;\gamma;a) = \exp\left[\dfrac{1}{2\pi\i}\displaystyle\int\limits_{\mathbb{R}}
\frac{\ln\left\{\frac{s^2+a^2}{s^2+\kappa_{\gamma}^2}\(1-\e^{-\frac12(s^2+\kappa_{\gamma}^2)}\)
 \right\}\ \d s}{s-k}\right].
\]

Denote for further usage the coefficient of square of $\widehat\delta$ for $k\to\i\kappa_{\gamma},$
\begin{equation}\label{deltaser}\begin{split}
&\widehat\delta^2(k;\gamma;a) = \exp\(c_0 + \i c_1 (k-\i\kappa_{\gamma})
+c_2 (k-\i\kappa_{\gamma})^2 + \i c_3 (k-\i\kappa_{\gamma})^3 +  \ldots\),\quad k\to\i\kappa_{\gamma},
\\
&\widehat\delta^2(k;\gamma;a) = \exp\(-c_0 + \i c_1 (k+\i\kappa_{\gamma})
-c_2 (k+\i\kappa_{\gamma})^2 + \i c_3 (k+\i\kappa_{\gamma})^3 +  \ldots\),\quad k\to-\i\kappa_{\gamma},
\end{split}\end{equation}
where 
\[\begin{split}c_j = c_j(\gamma;a) &:= \dfrac{1}{\pi\i}\displaystyle\int\limits_{\Sigma_a}
\frac{\ln\left\{\frac{s^2+a^2}{s^2+\kappa_{\gamma}^2}\(1-\e^{-\frac12(s^2+\kappa_{\gamma}^2)}\)
 \right\}\ \d s}{(s-\i \kappa_{\gamma})^{j+1}}\in\mathbb{R},\quad j\ \mbox{ is even},\ j\geq 0,
 \\
 &:=\dfrac{-1}{\pi}\displaystyle\int\limits_{\Sigma_a}
\frac{\ln\left\{\frac{s^2+a^2}{s^2+\kappa_{\gamma}^2}\(1-\e^{-\frac12(s^2+\kappa_{\gamma}^2)}\)
 \right\}\ \d s}{(s-\i \kappa_{\gamma})^{j+1}}\in\mathbb{R},\quad j\ \mbox{ is odd},\ j\geq 0, 
 \end{split}\]
In the case $\gamma = 1$ the limits in \eqref{deltaser} should be understood in the sense $k\to 0, \Im k>0$ and $k\to 0, \Im k<0,$ respectively.

Since the r.h.s. of \eqref{dh} does not vanish in the layer $|\Im k|<a,$ the logarithm in the latter integral is well-defined not only on the real line, but also in the above mentioned layer. 

Hence, when computing  $\widehat\delta(k;\gamma;a)$ numerically
at the point $\i\kappa_{\gamma}$ for $\gamma = 1$ or $\gamma$ close to 1, we can deform the contour of integration, integrating instead over the contour \\$\Sigma_{a} = (-\infty, -a/4)\cup(-\frac{a}{4}, -\i\frac{a}{4})\cup(-\i \frac{a}{4}, \frac{a}{4})\cup (\frac{a}{4},+\infty).$

Furthermore, since the r.h.s. in \eqref{dh} is uniformly continuous as $\gamma\to1-0$ and non-vanishing, the function $\widehat\delta(k;\gamma)$ is also uniformly continuous as $\gamma\to1-0.$
This means that 
\[\widehat\delta(k;\gamma;a)\to\widehat\delta(k;1;a)\quad \mbox{ as}\quad \gamma\to1-0\quad \mbox{ uniformly w.r.t. }\ k\in\mathbb{C}.\]
The function $L(k;\gamma)$ can be written with the help of function $\widehat\delta(k;\gamma;a)$ as follows:
\begin{equation}\label{Ldelta}
\begin{split}
L(k;\gamma) &= \frac{\e^{-\frac14(k^2+\kappa_{\gamma}^2)}}{1-\e^{-\frac12(k^2+\kappa_{\gamma}^2)}}\cdot
\frac{\(k+\i\kappa_{\gamma}\)^2}{\(k+\i a\)^2}
\cdot
\widehat\delta^2(k;\gamma;a),\quad \Im k>0,
\\\\
&= \e^{-\frac14(k^2+\kappa_{\gamma}^2)}
\(1-\e^{-\frac12(k^2+\kappa_{\gamma}^2)}\)
\cdot
\frac{\(k-\i a\)^2}{\(k-\i\kappa_{\gamma}\)^2}
\cdot
\widehat\delta^2(k;\gamma;a),\quad \Im k<0.
\end{split}
\end{equation}
From this representation we see that indeed, for $\gamma<1,$
 in $\Im k>0$ there is a simple pole at $k = \i \kappa_{\gamma},$ and in $\Im k<0$ there is a simple zero at $k = -\i\kappa_{\gamma}.$
 Furthermore, for $\gamma=1,$ $\kappa_{1} = 0,$ the $L(k;1)$ does not have neither zero nor pole at $k = 0.$
 
 Furthermore, expanding \eqref{Ldelta} into series for $k\to\i\kappa_{\gamma},$
 for $\gamma<1$ we obtain
 \[L(k;\gamma) =:-\(
 \frac{L_{-1}(\gamma)}{k-\i\kappa_{\gamma}}
 +L_0(\gamma)+L_1(\gamma)\kappa_{\gamma}^2+\ldots\):= \dfrac{-4\i \e^{c_0(\gamma;a)}}{(a+\kappa_{\gamma})^2\ (k-\i\kappa_{\gamma})} + \mathcal{O}(1), \quad k \to\i\kappa_{\gamma},\]
 whence 
 \begin{equation}\label{Lm1gamma}
 L_1(\gamma) = \frac{4\i \e^{c_0(\gamma;a)}}{(a+\kappa_{\gamma})^2}.
 \end{equation}
 For $\gamma = 1,$ $\kappa_1 = 0,$ both expressions in \eqref{Ldelta} must give the same series at $k\to0.$ Thus,
 \begin{equation}\label{c0}\begin{split}
 &
\e^{c_0(1;a)} =  \frac{a^2}{2},
\quad 
c_2(1;a) = -\frac14+\frac{1}{a^2},
\\
&
L(k;1) = -\(1+\i\(c_1(1;a)+\frac{2}{a}\)k
-
\( \frac14+ \frac12 \(c_1(1;a)+\frac2{a}\)^2\)k^2+\mathcal{O}(k^3)\).
\end{split} \end{equation}
 \end{proof}
 \begin{rem}
 Let us mention, that for $\gamma = 1,$ we have locally as $k\to 0$
\[\begin{split}&
T(k;1) = \frac{-\i k}{\sqrt{2}} + \mathcal{O}(k^2),\quad k\to 0.
\\
&L(k;1)=-\i \(1+\i L_1(1) k+\mathcal{O}(k^2)\), k\to 0,
\qquad \mbox{ and } \quad L_1(1)\in\mathbb{R}.
\end{split}\]
The fact that $L_1(1)$ is real follows from $|L(k;\gamma)|^2<1$ for $k\in\mathbb{R}.$
Furthermore, for $\gamma<1,$ we have $R(\i\kappa_{\gamma};\gamma) = -\i,$ and
\[L_{-1}(\gamma) = \frac{1}{\kappa_{\gamma}}T(\i\kappa_{\gamma};\gamma)^2>0.\]
Let us also mention another formula for $L_{-1}(\gamma),$ which can be derived from the previous ones,
\[
\ln\frac{L_{-1}(\gamma)}{2\kappa_{\gamma}\cdot2\kappa^2_{\gamma}}
=
\frac{1}{\pi\i}\int\limits_{-\infty}^{+\infty}
\frac{\ln\(1-\e^{-\frac12\kappa_{\gamma}^2(1+s^2)}\)}{s-\i}\d s.
\]
It follows from \eqref{Lm1gamma} and the first of the formulas \eqref{c0} that
\[\lim\limits_{\gamma\to 1-0}\frac{L_{-1}(\gamma)}{2\kappa_{\gamma}} = 
\frac{2}{a^2}\exp\left\{\frac{1}{\pi\i}\int\limits_{\Sigma_{a}}\frac{\ln\([1-\e^{-\frac{s^2}{2}}]\frac{s^2+a^2}{s^2}\)\d s}{s}\right\}=1.\]
\end{rem}

\subsection{Long $x\to-\infty$ analysis for $\gamma<1,$ and proof of (c)}
Since we are mostly interested in $\gamma$ that are close to 1, we restrict here our attention to $\gamma \in (\frac{1}{\sqrt{\e}},1)\approx(0.6065,1)$ (for $\gamma<1/\sqrt{\e}$ the analysis can be done in a more simple fashion).
For such $\gamma,$ we have $\kappa_{\gamma}<1$ and hence
the point $\i \kappa_{\gamma}$ lies in the domain $\Omega_1.$

\noindent 
Define a function
\begin{equation}\label{N}
\N(x;k;\gamma) = \M(x;k;\gamma)\cdot
\begin{cases}
T(k;\gamma)^{-\sigma_3}\begin{pmatrix}1 & \frac{\ol{R(\ol k;\gamma)}\ T(k;\gamma)^2}{1-R(k;\gamma)\ol{R(\ol k;\gamma)}}
\,\e^{-2\i k x} \\ 0 & 1\end{pmatrix},
k\in\Omega_1,
\\
\ol{T(\ol k;\gamma)}^{\,\sigma_3}\begin{pmatrix}1 & 0 \\ \frac{R(k;\gamma)\ \ol{T(\ol k;\gamma)}^{\,2}}{1-R(k;\gamma)\ol{R(\ol k;\gamma)}}\,\e^{2\i k x} & 1\end{pmatrix},
k\in\Omega_2,
\\
T(k;\gamma)^{-\sigma_3}, k \in \Omega_3,
\\
\ol{T(\ol k;\gamma)}^{\, \sigma_3}, k\in\Omega_4.
\end{cases}
\end{equation}
The function $\N(x;k;\gamma)$ solves the following RHP.
\begin{RHP}\label{NegxRHP}
To find a $2\times2$ matrix-valued function $\N(x;k;\gamma)$ that satisfies the following properties:
\begin{itemize}
\item analyticity: $\N(x,t;k;\gamma)$ is meromorphic in $k\in\(\mathbb{C}\setminus\Sigma\),$ with simple poles at $k=\pm\i\kappa_{\gamma}$,\\
and continuous up to the boundary $k\in\Sigma=L_1\cup L_2;$
\item jumps:
 $\N_-(x;k;\gamma)=\N_+(x;k;\gamma)\J_{N}(x;k;\gamma),$ where
\[\J_N=\begin{pmatrix}1 & \frac{\ol{R(\ol k;\gamma)}\ T(k;\gamma)^2}{1-R(k;\gamma)\ol{R(\ol k;\gamma)}}
\,\e^{-2\i k x} \\ 0 & 1\end{pmatrix}, k\in L_1,
\qquad
\J_N = \begin{pmatrix}1 & 0 \\ \frac{-R(k;\gamma)\ \ol{T(\ol k;\gamma)}^{\,2}}{1-R(k;\gamma)\ol{R(\ol k;\gamma)}}\,\e^{2\i k x} & 1\end{pmatrix}
 \ k\in L_2;\]
\item poles at $k=\pm\i\kappa_{\gamma}$:
function
\[
\N(x;k;\gamma)\begin{pmatrix}1 & \frac{-\ol{R(\ol k;\gamma)}\ T(k;\gamma)^2}{1-R(k;\gamma)\ol{R(\ol k;\gamma)}}
\,\e^{-2\i k x} \\ 0 & 1\end{pmatrix}
\]is regular at $k=\kappa_{\gamma},$
\\
function
\[\N(x;k;\gamma)\begin{pmatrix}1 & 0 \\ \frac{-R(k;\gamma)\ \ol{T(\ol k;\gamma)}^{\,2}}{1-R(k;\gamma)\ol{R(\ol k;\gamma)}}\,\e^{2\i k x} & 1\end{pmatrix}\]
is regular at $k=-\i\kappa;$
\item asymptotics at the infinity:
\[\N(x;k;\gamma)\to I  = \begin{pmatrix}1&0\\0&1\end{pmatrix}\quad \mbox{ as } \quad k\to\infty.\]
\end{itemize}
\end{RHP}

\begin{rem}
To see that the above RHP is indeed well-posed, we can rewrite the pole conditions as jump conditions across some circles of small radius $\varepsilon_{\gamma}<\frac{\kappa_{\gamma}}{3}$ around the points $\pm\i\kappa.$
To this end, define
\[\begin{split}
\N_{reg}(x;k;\gamma) &= N(x;k;\gamma)\begin{pmatrix}1 & \frac{-\ol{R(\ol k;\gamma)} \ T(k;\gamma)^2}{1-R(k;\gamma)\ol{R(\ol k;\gamma)}}
\,\e^{-2\i k x} \\ 0 & 1\end{pmatrix},\quad |k-\i\kappa_{\gamma}|<\varepsilon_{\gamma}, 
\\&=\N(x;k;\gamma)\begin{pmatrix}1 & 0 \\ \frac{-R(k;\gamma)\ \ol{T(\ol k;\gamma)}^{\,2}}{1-R(k;\gamma)\ol{R(\ol k;\gamma)}}\,\e^{2\i k x} & 1\end{pmatrix},
\quad |k+\i\kappa_{\gamma}|<\varepsilon_{\gamma}
\\
& =\N(x;k;\gamma),\ \mbox{elsewhere}.
\end{split}\]
Then $\N_{reg}$ solves the RHP for $\N,$ with the pole conditions being replaced by the jump conditions
\[\begin{split} & \N_{reg,-}(x;k;\gamma) = \N_{reg,+}(x;k;\gamma)\begin{pmatrix}1 & \frac{\ol{R(\ol k;\gamma)}\ T(k;\gamma)^2}{1-R(k;\gamma)\ol{R(\ol k;\gamma)}}
\,\e^{-2\i k x} \\ 0 & 1\end{pmatrix}, \quad k\in C_{\varepsilon_{\gamma}}(\i\kappa_{\gamma}),
\\
&
\N_{reg,-}(x;k;\gamma) = \N_{reg,+}(x;k;\gamma)
\begin{pmatrix}1 & 0 \\ \frac{R(k;\gamma)\ \ol{T(\ol k;\gamma)}^{\,2}}{1-R(k;\gamma)\ol{R(\ol k;\gamma)}}\,\e^{2\i k x} & 1\end{pmatrix}, \quad k\in C_{\varepsilon_{\gamma}}(-\i\kappa_{\gamma}),
\end{split}\]
where by $C_{r}(a)$ we denote the circle with the center $a$ and radius $r,$ oriented counter-close-wise, so that the positive side of the contour is inside the circle.
\end{rem}

\subsection*{Model problem $\N_{mod}(x;k;\gamma).$}
We see that the jumps for the $\N(x;k;\gamma)$ are exponentially close to $I$ as $x\to-\infty.$ This suggests that the main contribution to the asymptotics of $\N(x;k;\gamma)$ comes from the pole conditions at the points $k=\pm\i\kappa_{\gamma}.$
Introduce an anzatz
\begin{equation}\label{Nmod}
\N_{mod}(x;k;\gamma) = \begin{bmatrix}1+\frac{\i \alpha(x;\gamma)}{k+\i\kappa_{\gamma}}
&\frac{\i \beta(x;\gamma)}{k-\i\kappa_{\gamma}}
\\
\frac{-\i \beta(x;\gamma)}{k+\i\kappa_{\gamma}}
&
1-\frac{\i \alpha(x;\gamma)}{k-\i\kappa_{\gamma}}
\end{bmatrix}
\end{equation}
with real $\alpha(x;\gamma),$ $\beta(x;\gamma)$ which are to be determined from the condition that $\N_{mod}$ satisfies the pole conditions of the RHP \ref{NegxRHP}.
Then the error matrix
$\N_{err}(x;k;\gamma) = \N(x;k;\gamma)\N_{mod}(x;k;\gamma)^{-1}$ 
will be regular at the points $k = \pm\i\kappa_{\gamma}$ (the simplest way to see this is to rewrite again the pole conditions as jump conditions), and the jumps for it will be exponentially close to $I$ (smaller than $\e^{-C |x|}$ for any $C>0,$ which we can achieve by moving the contours $L_1,$ $L_2$ towards $\pm\i\infty$),  provided that $A,B$ are uniformly bounded.
Hence, we would see that indeed $\N_{mod}(x;k;\gamma)$ is close to $\N(x;k;\gamma)$ for $x\to-\infty.$

Substituting the ansatz \eqref{Nmod} into pole condition at $k=\i\kappa_{\gamma}$ of RHP \ref{NegxRHP} (just one of the condition suffices in view of symmetries), and recalling the definition \eqref{L}, \eqref{Lstar} of the {\it left reflection coefficient} $L(k;\gamma),$ we obtain the following conditions for $\alpha(x;\gamma),$ $\beta(x;\gamma)$:
\[
\begin{cases}
\(1+\dfrac{\alpha(x;\gamma)}{2\kappa_{\gamma}}\)L_{-1}(\gamma)\e^{2\kappa_{\gamma}x} - \beta(x;\gamma) = 0,
\\\\
\dfrac{\beta(x;\gamma)}{2\kappa_{\gamma}} L_{-1}(\gamma)\e^{2\kappa_{\gamma} x} - \alpha(x;\gamma) = 0,
\end{cases}
\qquad
\mbox{ whence }
\quad 
\begin{cases}
\alpha(x;\gamma) = \dfrac{2\kappa_{\gamma}\e^{4 x \kappa_{\gamma}}{L_{-1}(\gamma)}^2}{4\kappa_{\gamma}^2 - \e^{4x\kappa_{\gamma}}{L_{-1}(\gamma)}^2},
\\\\
\beta(x;\gamma) = \dfrac{4\kappa_{\gamma}^2\e^{2 x \kappa_{\gamma}}{L_{-1}(\gamma)}}{4\kappa_{\gamma}^2 - \e^{4x\kappa_{\gamma}}{L_{-1}(\gamma)}^2}.
\end{cases}
\]
We see that indeed $\alpha, \beta$ are bounded as $x\to-\infty,$ and both of them are positive. One can check that for such choice of $\alpha, \beta$, we have $\det N_{mod}(x;k;\gamma) \equiv 1.$
We have 
\[
\lim\limits_{k\to\i\infty}k(\N(x;k;\gamma)-I)=
\lim\limits_{k\to\i\infty}k(\M(x;k;\gamma)T(x;k;\gamma)^{-\sigma_3}-\1) =
\begin{pmatrix}\i A_1(x;\gamma)+\i T_1(\gamma) & \i B_1(x;\gamma) \\ -\i B_1(x;\gamma) & -\i A_1(x;\gamma)-\i T_1(\gamma)\end{pmatrix}, 
\]
and hence 
\[
A_1(x;\gamma) = -T_1(\gamma)+\alpha(x;\gamma)+\mathcal{O}(\e^{-C|x|}),
\qquad 
B_1(x;\gamma) = \beta(x;\gamma)+\mathcal{O}(\e^{-C|x|}),
\]
for any $C>0.$
This finishes proof for (c).

\subsection{Long $x\to-\infty$ analysis for $\gamma = 1,$ and proof of (d)}
Here we again define function $N(x;k;\gamma=1)$ by formula \eqref{N}.
The function $\N(x;k;1)$ solves the following RHP.
\begin{RHP}\label{NegxRHP1}
To find a $2\times2$ matrix-valued function $\N(x;k;1)$ that satisfies the following properties:
\begin{itemize}
\item analyticity: $\N(x;k;1)$ is meromorphic in $k\in\mathbb{C}\setminus{\Sigma},$ with a simple pole at $k=0$,\\
and continuous up to the boundary $k\in\Sigma=L_1\cup L_2;$
\item jumps:
 $\N_-(x;k;1)=\N_+(x;k;1)\J_{N}(x;k;1),$ where
\[\J_N=\begin{pmatrix}1 & \frac{\ol{R(\ol k;1)}\ T(k;1)^2}{1-R(k;1)\ol{R(\ol k;1)}}
\,\e^{-2\i k x} \\ 0 & 1\end{pmatrix}, k\in L_1,
\qquad
\J_N = \begin{pmatrix}1 & 0 \\ \frac{-R(k;1)\ \ol{T(\ol k;1)}^{\,2}}{1-R(k;1)\ol{R(\ol k;1)}}\,\e^{2\i k x} & 1\end{pmatrix}
 \ k\in L_2;\]
\item singularity at $k=0:$
function
\[
\N(x;k;1)\begin{pmatrix}1 & \frac{-\ol{R(\ol k;1)}\ T(k;1)^2}{1-R(k;1)\ol{R(\ol k;1)}}
\,\e^{-2\i k x} \\ 0 & 1\end{pmatrix}k^{\sigma_3}
\]is regular at $k=0,$ $\Im k \geq0,$
\\
function
\[\N(x;k;1)\begin{pmatrix}1 & 0 \\ \frac{-R(k;1)\ \ol{T(\ol k;1)}^{\,2}}{1-R(k;1)\ol{R(\ol k;1)}}\,\e^{2\i k x} & 1\end{pmatrix}k^{-\sigma_3}\]
is regular at $k=0,$ $\Im k\leq 0;$
\item asymptotics at the infinity:
\[\N(x;k;1)\to \1  = \begin{pmatrix}1&0\\0&1\end{pmatrix}\quad \mbox{ as } \quad k\to\infty.\]
\end{itemize}
\end{RHP}

\begin{rem}
To see that the above RHP is indeed well-posed, we can rewrite the pole conditions as jump conditions across a circle of a small radius $\varepsilon$ around the point $0,$ and the segment $(-\varepsilon, \varepsilon).$
To this end, define 
\[\begin{split}
\N_{reg}(x;k;1) &= N(x;k;1)\begin{pmatrix}1 & \frac{-\ol{R(\ol k;1)} \ T(k;1)^2}{1-R(k;1)\ol{R(\ol k;1)}}
\,\e^{-2\i k x} \\ 0 & 1\end{pmatrix}T(k;1)^{\sigma_3},\quad |k|<\varepsilon, \Im k > 0
,
\\&=\N(x;k;1)\begin{pmatrix}1 & 0 \\ \frac{-R(k;1)\ \ol{T(\ol k;1)}^{\,2}}{1-R(k;1)\ol{R(\ol k;1)}}\,\e^{2\i k x} & 1\end{pmatrix} \ol{T(\ol k;1)}^{\,-\sigma_3},
\quad |k|<\varepsilon, \Im k<0,
\\
& =\N(x;k;1),\ \mbox{elsewhere}.
\end{split}\]
Then $\N_{reg}$ solves the RHP for $N,$ with the singularity condition being replaced by the jump conditions
\[\begin{split} & \N_{reg,-}(x;k;\gamma) = \N_{reg,+}(x;k;\gamma)T(k;1)^{-\sigma_3}\begin{pmatrix}1 & \frac{\ol{R(\ol k;\gamma)}\ T(k;\gamma)^2}{1-R(k;\gamma)\ol{R(\ol k;\gamma)}}
\,\e^{-2\i k x} \\ 0 & 1\end{pmatrix}, \quad k\in C^+_{\varepsilon}(0),
\\
&
\N_{reg,-}(x;k;\gamma) = \N_{reg,+}(x;k;\gamma)\ol{T(\ol k;1)}^{\,\sigma_3}
\begin{pmatrix}1 & 0 \\ \frac{R(k;\gamma)\ \ol{T(\ol k;\gamma)}^{\,2}}{1-R(k;\gamma)\ol{R(\ol k;\gamma)}}\,\e^{2\i k x} & 1\end{pmatrix}, \quad k\in C^-_{\varepsilon}(0),
\\
&
\N_{reg,-}(x;k;1) = \N_{reg,+}(x;k;1)
\begin{pmatrix}1 & \ol{R(k;1)}\e^{-2\i kx} \\ -R(k;1)\e^{2\i kx} & 1 - |R(k;1)|^2\end{pmatrix}, k\in(-\varepsilon,\varepsilon),
\end{split}\]
where by $C^+_{\varepsilon}(0)$ we denote part of the oriented counter-clock-wise circle $C_{\varepsilon}(0),$ which lies in $\Im k>0,$ and similar for $C^-_{\varepsilon}.$
\end{rem}

\subsection*{Model problem $\N_{mod}(x;k;1).$}
We see that the jumps for the $\N(x;k;1)$ are exponentially close to $I$ as $x\to-\infty.$ This suggests that the main contribution to the asymptotics of $\N(x;k;1)$ comes from the singularity condition at the point $k=0.$
Introduce an anzatz
\begin{equation}\label{Nmod1}
\N_{mod}(x;k;1) = \begin{bmatrix}1+\dfrac{\i \alpha(x;1)}{k}
&\dfrac{\i\beta(x;1)}{k}
\\\\
\dfrac{-\i\beta(x;1)}{k}
&
1-\dfrac{\i \alpha(x;1)}{k}
\end{bmatrix}
\end{equation}
with real $A(x;1),$ $B(x;1)$ which are to be determined from the condition that $N_{mod}$ satisfies the singularity conditions of the RHP \ref{NegxRHP1}.
Then the error matrix 
$N_{err}(x;k;1) = N(x;k;1)N_{mod}(x;k;1)^{-1}$ 
will be regular at the points $k =0$ (the simplest way to see this is to rewrite again the pole conditions as jump conditions), and the jumps for it will be exponentially close to $\1$ (smaller than $\e^{-C |x|}$  for any $C>0,$ which we can achieve by moving the contours $L_1,$ $L_2$ towards $\pm\i\infty$), provided that $A,B$ are uniformly bounded.
Hence, we would see that indeed $\N_{mod}(x;k;1)$ is close to $\N(x;k;1)$ for $x\to-\infty.$

Substituting the ansatz \eqref{Nmod1} into singularity condition at $k=0,$ $\Im k>0$ of RHP \ref{NegxRHP} (just one of the condition suffices in view of symmetries), and recalling the definition \eqref{L}, \eqref{Lstar} of the {\it left reflection coefficient} $L(k;1),$ we obtain the following conditions for $\alpha(x;1),$ $\beta(x;1)$: 
\[
\begin{cases}
\alpha(x;1) = \beta(x;1),\\
\alpha(x;1)(L_1(1)-2x)=1
\end{cases}
\qquad
\mbox{ whence }
\quad 
\alpha(x;1)  = \beta(x;1) =\frac{1}{-2x+L_1(1)}.
\]
We see that indeed $\alpha, \beta$ are bounded as $x\to-\infty,$ and both of them are positive. One can check that for such choice of $\alpha, \beta$, we have $\det \N_{mod}(x;k;1) \equiv 1.$
We have 
\[
\lim\limits_{k\to\i\infty}k(\N(x;k;1)-I)=
\lim\limits_{k\to\i\infty}k(\M(x;k;1)T(x;k;1)^{-\sigma_3}-I) =
\begin{pmatrix}\i A_1(x;1)+\i T_1(1) & \i B_1(x;1) \\ -\i B_1(x;1) & -\i A_1(x;1)-\i T_1(1)\end{pmatrix}, 
\]
and hence 
\[
A_1(x;1) = -T_1(1)+\alpha(x;1)+\mathcal{O}(\e^{-C|x|}),
\qquad 
B_1(x;1) = \beta(x;1)+\mathcal{O}(\e^{-C|x|}),
\]
for any $C>0.$
This finishes proof for (d).

\section{Expression for $\int_{x}^{+\infty}y(z)\d z,$ and proof of (e)}\label{sectE}
Expand the solution $\M(x;k;\gamma)$ of the original RHP at $k=0,$
\[\M(x;k;\gamma) = \M_0(x;\gamma) + k \M_1(x;\gamma) + 
k^2 \M_2(x;\gamma) + k^3 \M_3(x;\gamma) + \ldots,\]
and substitute this into the differential equation for $\M:$
\[\M_x(x;k;\gamma) + \i k [\s_3, \M(x;k;\gamma)] = \mathbf{Q}(x;\gamma)\M(x;k;\gamma).\]
Comparing the elements of $k^0, k^1,\ldots,$ we obtain
\[\partial \M_0(x;\gamma) = \mathbf{Q}(x;\gamma)\M_0(x;k;\gamma),
\qquad 
\partial \M_1(x;\gamma) + [\s_3, \M_0(x;k;\gamma)] = 
\mathbf{Q}(x;\gamma)\M_1(x;k;\gamma), \ldots.\]
Let us treat the first one.
Denote, using symmetries (Lemma \eqref{lem_sym}), \[\M_0(x;\gamma) = \begin{pmatrix} r(x;\gamma) & w(x;\gamma) \\ w(x;\gamma) & r(x;\gamma)\end{pmatrix},\]
then
we obtain
\[
\begin{split}
r_x(x;\gamma) = -q(x;\gamma) w(x;\gamma),
\qquad
w_x(x;\gamma) = -q(x;\gamma) r(x;\gamma),
\end{split}
\]
and the boundary conditions are
\[\lim_{x\to+\infty}r(x;\gamma) = 1,
\qquad 
\lim_{x\to+\infty}w(x;\gamma) = 0.
\]
We obtain
\[(r+w)_x=-q(x)(r+w),\qquad (r-w)_x=q(x)(r-w)\]
whence
\[
r+w=\exp\(\int_{x}^{+\infty}q(z;\gamma)\d z\),
\qquad
r-w=\exp\(-\int_{x}^{+\infty}q(z;\gamma)\d z\),
\]
and finally
\[
r(x;\gamma)=\cosh\(\int_{x}^{+\infty}q(z;\gamma)\d z\),\qquad
w(x;\gamma)=\sinh\(\int_{x}^{+\infty}q(z;\gamma)\d z\).
\]

\noindent 
Now substitute this in the ingredients of the asymptotic analysis.
We have
\[
\P=\begin{cases}
\begin{pmatrix}\M_1-R\e^{2\i k x}\M_2, & \M_2\end{pmatrix},\quad k\in\Omega_1,
\\
\begin{pmatrix}\M_1,& \M_2-\ol{R}\e^{-2\i k x}\M_1\end{pmatrix},
\quad k\in\Omega_2,
\end{cases}
\]
and \[\P(k) = \begin{pmatrix}r(x;\gamma) & w(x;\gamma) \\
w(x;\gamma) & r(x;\gamma) \end{pmatrix}+\mathcal{O}(k),\quad k\to 0.\]
Now set $\gamma = 1.$
We have for $x\to-\infty,$
\[
\N=\begin{cases}
\begin{pmatrix}\frac{1}{T}\M_1, & T\M_2-\frac{1}{T}\M_1L\e^{-2\i k x}
\end{pmatrix},\quad k\in\Omega_1,
\\
\begin{pmatrix}\frac{1}{T}\M_1 - T\M_2\ol{L}\e^{2\i k x}, & T\M_2
\end{pmatrix},\quad k\in\Omega_2,
\end{cases}
\]
and 
\[\N(k)\sim \begin{pmatrix}1+\frac{\i \alpha}{k} & \frac{\i\beta}{k} \\\frac{-\i \beta}{k} & 1-\frac{\i \alpha}{k}\end{pmatrix}.\]
Hence,
\[\begin{cases}
(\M_1-R\e^{2\i k x}\M_2)_{[1]} \sim T(1+\frac{\i \alpha}{k}) -R\e^{2\i k x}\frac{1}{T}\(\frac{\i\beta}{k}+L\e^{-2\i k x}(1+\frac{\i \alpha}{k})\)
=r(x;1)+\mathcal{O}(k), k\to 0,
\\\\
(\M_2)_{[1]}\sim\frac{1}{T}\(\frac{\i \beta}{k}+L\e^{-2\i k x}(1+\frac{\i \alpha}{k})\)
=
w(x) + \mathcal{O}(k), \ k\to 0,
\end{cases}\]

Expanding the middle term, we see that the $k^{-2}$ term vanish because of $\alpha(x;1) = \beta(x;1),$ and the  $k^{-1}$ term vanish because
$\alpha(x;1)=\beta(x;1) = \frac{1}{-2x+L_1(1)}.$
Then, comparing the $k^0$ terms gives us 
\[
\begin{split}
&
r(x) = \frac{\(L_1(1)-2x\)^2+2L_2(1)+L_1(1)^2+1}{\sqrt{2}\(L_1(1)-2x\)},
\\
&w(x) = - \frac{\(L_1(1)-2 x\)^2+L_1(1)^2+2L_2(1)}{\sqrt{2}\(L_1(1)-2x\)},
\end{split}\]
were we denoted (\eqref{L2L1})
\[\begin{split}
T(k;1) = \frac{-\i k}{\sqrt{2}} + T_2 k^2 +T_3 k^3 +\ldots,
\quad
L(k;1) = -\(1 + \i L_1(1) k + L_2(1) k^2 +\L_3 k^3 +\ldots\).
\end{split}\]
and $r^2-w^2=1$ gives us $L_1^2(1)+2L_2(1)=-\frac12.$
Hence,
\[\begin{split}
&\exp\(\int\limits_{x}^{+\infty}q(z;1)\d z\)
=r(x)-w(x) = \sqrt{2}\(L_1(1)-2x\),
\\
&
\exp\(-\int\limits_{x}^{+\infty}q(z;1)\d z\)
=r(x) + w(x) = 
\frac{1}{\sqrt{2}(L_1(1)-2x)}.
\end{split}\]

This proves (e).

\section{Some conservative quantities.}

%
%

\begin{lemma}\label{lem_cons}
The quantities \eqref{cons} do not depend on time $t,$ and the quantity $M(1)$ does not depend on $x<0.$
%
%
%
\end{lemma}

\begin{proof}First of all, the integrals are convergent\footnote{For $t\neq 0$ there are asymptotic formulas similar to \eqref{asympKdV}. Because $\nu(\xi)$ is exponentially small for $x\to-\infty,$ the $q(x,t)$ is also exponentially small, and hence integrals are convergent.}. Then it is enough to differentiate w.r.t. $t$ and $x,$ using \eqref{MKdVKdV}. Let us consider for example $(\ref{cons}c)$ and $(\ref{cons}d).$

(\ref{cons}c). 
Since \[q_t-6q^2q_x+q_{xxx}=0,\]we have 
\[\partial_t M = \(x\(3q^4-2qq_{xx}+q_x^2\)+2qq_x\)\Big|_{-\infty}^{+\infty}
-3\int\limits_{-\infty}^{+\infty}(q^4+q_x^2)\d x + N = 0.\]

\eqref{M1} The integral converges. 
Since $\partial_x M(1) = 0,$
we conclude that it does not depend on $x<0.$

\[\(-4y^6-4 y^3y_{xx} +12y^2y_x^2 +2 y_x y_{xxx} -
(y_xx)^2 \)_x = 4y^3y_{t}-2y_xy_{tx}.\]

%
%
%
%


\end{proof}

\newpage
\section{One-pager on Korteweg-de Vries equation}\label{sectKdV}
The KdV equation, $u=u(x,t),$
\[u_t(x,t)-6u(x,t)\,u_x(x,t)+u_{xxx}(x,t)=0\]
is the compatibility condition of the ordinary differential equations (Lax pair) for a function $f = f(x,t;k),$
\begin{equation}
\label{x-eqKdV}
(\ref{x-eqKdV}a):\quad -f_{xx} + u(x,t)  f = k^2 f, 
\qquad (\ref{x-eqKdV}b):\quad
f_t = 
 \(4k^2+u(x,t)\)\, f_x - \(4k+u_x(x,t)+c\) f, \nonumber
\addtocounter{equation}{1}
\end{equation}
where $c$ is an arbitrary constant.
Let $u_0(x) = u(x,0)$ be an initial function, $u_0(x)\to 0$ as $x\to\pm\infty.$ Let $f_{\pm}(x;k)$ be a solution of (\ref{x-eqKdV}a) for $t=0$ with asymptotics
\[f^+(x;k)=\e^{\i k x}(1+_\mathcal{O}(1)),\ x\to+\infty,
\qquad 
f^-(x;k)=\e^{-\i k x}(1+_\mathcal{O}(1)),\ x\to-\infty,\qquad k\in\mathbb{R}.\]
Define the spectral functions $a(k), b(k), R(k) = \frac{b(k)}{a(k)}$ by relations 
\[\begin{split}&
f^-(x;k) = a(k) \ol{f^+(x; k)} + b(k) f^+(x;k),\quad |a(k)|^2-|b|^2(k)=1,
\\
&2\i k a(k) = W(k) =  \left\{
f^-(x;k), f^+(x;k)
\right\},
\quad 
2\i k b(k) = \left\{
\ol{f^+(x;k)}, f^-(x;k)
\right\}
\end{split}
\]
where the brackets $\left\{f,g\right\} = fg_x-f_xg$ denote the Wronskian.
The $1\times 2$ vector function $V(x;k),$
\[
\begin{split}
\mathbf{V}(x;k) &= \(\frac{1}{a(k)} f_-(x;k)\e^{\i k x},\ f_+(x;k)\e^{-\i k x}\),\ \Im k>0,
\\
&= \(\ol{f_+(x;\ol k)}\e^{\i k x},\ \frac{1}{\ol{a(\ol k)}} \ol{f_-(x;\ol k)}\e^{-\i k x}\),\ \Im k<0,
\end{split}
\]
has the jump across the real line
\begin{equation}\label{JumpKdV}\mathbf{V}_-(x;k) = \mathbf{V}_+(x;k)\begin{pmatrix}
1 & \ol{R(\ol k)} \e^{-2\i k x}
\\
-R(k) \e^{2\i k x} & 1-|R(k)|^2
\end{pmatrix},\quad k\in\mathbb{R},\end{equation}
where $\mathbf{V}_{\pm}(x;k) = \mathbf{V}(x;k\pm\i 0)$ for real $k.$
Let $1\times 2$ vector-valued function $\mathbf{V}(x,t;k)$ has the jump \eqref{JumpKdV} with $2\i k x$ substituted with $2\i k x + 8\i k^3 t$ for all $t,$ and let $\mathbf{V}$ be normalized by the $(1,1)$ vector as $k\to\infty.$
Expanding $
f^{\pm}(x;k) = \e^{\pm\i k x}\(1+\frac{f^{\pm}_1}{k}+\frac{f^{\pm}_2}{k^2}+\ldots\),
$
for $k\to\infty,$ $\Im k\geq 0,$
and substituting into \eqref{x-eqKdV}, we find
\\
$
f^+_1 = \frac{\i}{2}\int\limits_{x}^{+\infty}u,\quad 
f^-_1 = \frac{\i}{2}\int\limits^{x}_{-\infty}u,
\quad
f_2^{\pm} = \frac12(f_1^{\pm})^2\pm\frac{\i}2 f_{1x}^{\pm},
\quad
a(k) = 1 - \frac{1}{2\i k}\int\limits_{-\infty}^{+\infty}u+\mathcal{O}(k^{-2}),
$
\begin{equation}\label{KdVu}
\mathbf{V}(x;k) = (1,1) + \frac{1}{2\i k }\int\limits_{x}^{+\infty}u(x)\d x
\begin{pmatrix}1,-1\end{pmatrix}+\mathcal{O}(k^{-2}),
\quad 
\mathbf{V}_{[1]}\mathbf{V}_{[2]} = 1 +\frac{u}{2k^2}+\mathcal{O}(k^{-3}),
\end{equation}
Function $u(x,t)$ obtained from $\mathbf{V}(x,t;k)$ by formulas \eqref{KdVu} is a solution of KdV for all $t.$
Furthermore,
define a singular at $k=0$ matrix
\[
\begin{split}
\M_{sing}(x;k) &= 
\frac12\begin{bmatrix}
\frac{1}{a(k)}\(f^--\frac{1}{\i k }f_x^-\)\e^{\i k x} & \(f^+-\frac{1}{\i k }f_x^+\)\e^{-\i k x}
\\
\frac{1}{a(k)}\(f^-+\frac{1}{\i k }f_x^-\)\e^{\i k x} & \(f^++\frac{1}{\i k }f_x^+\)\e^{-\i k x}
\end{bmatrix},\quad \Im k >0,
\\
&=\s_1 \ol{\M_{sing}(\ol k)}\s_1 = \s_1 \ol{\M_{sing}(-k)}\s_1 , \mbox{ where } \s_1 := \begin{pmatrix}0&1\\1&0\end{pmatrix}
\end{split}
\]
and a regular at $k=0$ matrix function
\[
\M_{reg}(x;k) = \begin{pmatrix}1 + \frac{\i \alpha(x)}{k} & \frac{\i \alpha(x)}{k} \\ \frac{-\i\alpha(x)}{k} & 1-\frac{\i\alpha(x)}{k}\end{pmatrix}\M_{sing}(x;k),\quad \mbox{ where } \alpha(x) = -\frac12\frac{f_x^+}{f^+}\Big|_{k=0}.
\]
Define 
\[q(x,t) = 2\alpha(x,t) := -\frac{\partial_x\mathbf{V}_{[2]}(x,t;k)}{\mathbf{V}_{[2]}(x,t;k)}\Big|_{k=\i 0}.\]
Then
\[u(x,t) = q^2(x,t)-q_x(x,t),\ \mbox{ and }\quad 
u_t-6uu_x+u_{xxx} = \(q_t-6q^2q_x+q_{xxx}\)_x.\]

\section{One-pager on modified Korteweg-de Vries equation}\label{sectMKdV}
The (defocusing) MKdV equation, $q=q(x,t),$
\[q_t-6q^2q_x+q_{xxx} = 0\]
is the compatibility condition of the system, $\mathbf{\Phi} = 
\mathbf{\Phi}(x,t;k),$
\begin{equation}\label{LaxMKdV}\mathbf{\Phi}_x + \i k\s_3 \mathbf{\Phi} = \mathbf{Q}(x,t)
\mathbf{\Phi},\quad 
\mathbf{\Phi}_t + 4\i k^3 \s_3 \mathbf{\Phi} = \mathbf{Q_2}(x,t;k)
\mathbf{\Phi},\end{equation}
where
\[\Q=\mathbf{Q}(x,t) = \begin{pmatrix}0&-q(x,t)\\-q(x,t) & 0\end{pmatrix},\quad 
\mathbf{Q}_2(x,t;k) = 4k^2 \Q-2\i(\Q_x+\Q^2)\s_3 k+(2\Q^3-\Q_{xx}).\]
Let $q_0(x) = q(x,0)\to 0$ as $x\to\pm\infty$ be an initial function. Let 
\[\mathbf{\Phi}^-(x,t;k)=\begin{pmatrix}\varphi^-(x;k) & \ol{\psi^-(x;\ol k)}\\
\psi^-(x,t;k) & \ol{\varphi^-(x;\ol k)}\end{pmatrix},
\quad
\mathbf{\Phi}^+(x,t;k)=\begin{pmatrix}\ol{\psi^+(x;\ol k)}
& \varphi^+(x;k)\\
\ol{\varphi^+(x;\ol k)} & \psi^+(x;k)\end{pmatrix}\]
be solutions of the first of \eqref{LaxMKdV}, normalized as
$\Phi^{\pm}(x;k)=\e^{-\i k x\sigma_3}(\1+\mathcal{O}(1)),\quad x\to\pm\infty, k\in\mathbb{R}.$ Define the transition matrix and spectral functions $a(k), b(k), R(k) = \frac{b(k)}{a(k)},$
\[T(k) = \(\Phi^+(x;k)\)^{-1}\Phi^-(x;k) =\begin{pmatrix}a(k) & \ol{b(\ol k)} \\ b(k) & \ol{a(\ol k)}\end{pmatrix},\qquad |a(k)|^2-|b|^2(k)=1.\]
The $2\times 2$ matrix-valued function
\[
\begin{split}
\M(x;k) &= \(\frac{1}{a(k)}\mathbf{\Phi}^-_{[1]}\e^{\i k x},\ \mathbf{\Phi}^+_{[2]}\e^{-\i k x}\),\quad \Im k>0,
\\
&=
\(\mathbf{\Phi}^+_{[1]}\e^{\i k x},\ \frac{1}{\ol{a(\ol k)}}\mathbf{\Phi}^-_{[2]}\e^{-\i k x}\),\quad \Im k<0,
\end{split}\]
has the jump 
\begin{equation}\label{JumpMKdV}\M_-(x;k)=\M_+(x;k)\begin{pmatrix}1 & \ol{R(k)}\e^{-2\i k x}
\\
-R(k)\e^{2\i k x} & 1-|R(k)|^2\end{pmatrix},\quad k\in\mathbb{R},\end{equation}
and symmetries $\s_1\M(k)\s_1 = \ol{\M(\ol k)} = \M(-k).$
Function $\M(x,t;k)$ which for all $t$ has the jump \eqref{JumpMKdV} with $2\i k x$ changed with $2\i k x+8\i k^3 t,$ and normalized to $\1$ as $k\to\infty,$ generates a solution $q(x,t)$ of MKdV by formulas
\begin{equation}\label{MKdVq}\begin{split}
&\M(x,t;k) = \1 + \frac{1}{2\i k}\begin{pmatrix}\int_x^{+\infty}q^2 & -q \\ q & -\int_{x}^{+\infty}q\end{pmatrix}\frac{1}{k}+\mathcal{O}(k^{-2}),\quad k\to\infty,
\\
&
\lim\limits_{k\to 0,\Im k>0}\M(x,t;k)\begin{pmatrix}1 & 0 \\ -R(k)\e^{2\i k x+8\i k^3 t} & 1\end{pmatrix}
=
\begin{pmatrix}\cosh\left[\int_x^{+\infty}q\right] & \sinh\left[\int_x^{+\infty}q\right]
\\
\sinh\left[\int_x^{+\infty}q\right]
&\cosh\left[\int_x^{+\infty}q\right]\end{pmatrix}.
\end{split}\end{equation}
If $q(x,t)$ is a solution to MKdV, then also $\widehat q(x,t) = -q(x,t)$ is a solution.The corresponding quantities 
\[\widehat{\mathbf{\Phi}^{\pm}} = \s_3\mathbf{\Phi}^{\pm}\s_3,\quad \widehat{a}(k)=a(k),\ \widehat{b}(k)=-b(k),\ \widehat{R}(k)=-R(k).\]
Functions $u(x,t) = q^2-q_x,$ $\widehat u = q^2+q_x$ are solutions to KdV, and the associated with $u$ spectral functions $a,b,R$ are the same as the ones associated with $q.$ The associated with $\widehat u$ spectral functions are the same as the ones associated with $\widehat q,$ i.e the reflection coefficient is the opposite.
If $R(0) = -1,$ then $u$ is a fast decaying solution of KdV and $\widehat u$ is a slowly decaying solution of KdV, and if $R(0)=1,$ then vice verse.

The $a_{MKdV}(k)\neq 0$ for $\Im k\geq 0,$ while $a_{KdV}$ might have simple zeros $a(\i\kappa)=0,$ for some $\kappa>0.$ In the latter case the corresponding solution $q$ of MKdV will have poles for real $x.$

\end{document}